\documentclass[11pt]{article}

\usepackage{amsmath,amssymb,mathrsfs,amsthm}
\usepackage{graphicx,cite,enumitem}
\usepackage{epstopdf}
\usepackage{xcolor}
\usepackage{pgfplots}
\usepackage{tikz}
\pgfplotsset{compat=1.14}
\numberwithin{equation}{section}
\usepackage{ulem}

\newtheorem{thm}{Theorem}[section]
\newtheorem{cor}[thm]{Corollary}

\newtheorem{rem}[]{Remark}
\newtheorem{assp}[thm]{Assumption}
\newtheorem{example}[thm]{Example}
\newtheorem{prop}[thm]{Proposition}

\DeclareMathOperator*{\argmin}{argmin}

\newcommand{\tri}[3]{{#1}_{#2}^{#3}}
\newtheorem{myDef}{Definition}

\def\E{{\mathbb E}}

\def\tr{{\rm{tr}}}

\def\N{{\mathcal N}}
\newcommand{\Y}{\mathcal Y}
\newcommand{\X}{\mathcal X}

\begin{document}

	\title{On the asymptotical regularization for linear inverse problems in presence of white noise}
	\author{Shuai Lu\thanks{Shanghai Key Laboratory for Contemporary Applied Mathematics, Key Laboratory of Mathematics for Nonlinear Sciences and School of Mathematical Sciences, Fudan University, 200433 Shanghai, China (Email: slu@fudan.edu.cn).}~, Pingping Niu\thanks{Shanghai Key Laboratory for Contemporary Applied Mathematics, Key Laboratory of Mathematics for Nonlinear Sciences and School of Mathematical Sciences, Fudan University, 200433 Shanghai, China (Email: ppniu14@fudan.edu.cn).}~ and Frank Werner\thanks{Institut f\"ur Mathematik, University of Wuerzburg, Emil-Fischer-Str. 30, 97074 W\"urzburg (Email: frank.werner@mathematik.uni-wuerzburg.de)}~\thanks{Corresponding author}}
	
	\date{\today}
	\maketitle
	
	\abstract{
		We interpret steady linear statistical inverse problems as artificial dynamic systems with white noise and introduce a stochastic differential equation (SDE) sytem where the inverse of the ending time $T$ naturally plays the role of the squared noise level. The time-continuous framework then allows us to apply classical methods from data assimilation, namely the Kalman-Bucy filter and 3DVAR, and to analyze their behavior as a regularization method for the original problem. Such treatment offers some connections to the famous asymptotical regularization method, which has not yet been analyzed in the context of random noise. We derive error bounds for both methods in terms of the mean-squared error under standard assumptions and discuss commonalities and differences between both approaches. If an additional tuning parameter $\alpha$ for the initial covariance is chosen appropriately in terms of the ending time $T$, one of the proposed methods gains order optimality. Our results extend theoretical findings in the discrete setting given in the recent paper Iglesias et al. \cite{ILLS2017}. Numerical examples confirm our theoretical results.
	}
	\begin{quote}
		\noindent
		{\small \textbf{Keywords:}
			Statistical inverse problems,
			data assimilation,
			Kalman-Bucy filter,
			asymptotical regularization,
			convergence rates}
	\end{quote}
	
	\begin{quote}
		\noindent
		{\small \textbf{AMS-classification (2020):}
			65J20, 47A52, 62M20
		}
	\end{quote}
	
	\section{Introduction}\label{se1}
	
	\subsection{From steady inverse problems to dynamical systems}
	
	The probably most often investigated setting in statistical inverse problems is the recovery of an unknown solution $u^{\dag}$ from the indirect noisy measurement
	\begin{align}\label{eq_LInP}
	y^{\delta} = A u^{\dag} + \delta \eta
	\end{align}
	where $A$ is a compact linear operator acting between separable Hilbert spaces $\X$ and $\Y$, $\eta$ is a (weak) Gaussian process on $\Y$ with a covariance operator $\Sigma$ (we write $\eta \sim \N_\Y\left(0, \Sigma\right)$), and $\delta > 0$ is a noise level. Note that (depending on $\Sigma$), $y^\delta$ might not be identifiable with an element in $\Y$, but has rather to be understood as an element in $\Y^*$, which is why \eqref{eq_LInP} is usually read in a weak sense. \eqref{eq_LInP} is a prototypical inverse problem, which has widely been considered in the literature, and we refer to the monographs \cite{EHN1996,LP2013} and the references therein.
	
	To solve the linear inverse problem \eqref{eq_LInP} stably, one usually uses regularization methods. Most common examples are either spectral methods of the form
	\[
	u_\epsilon^\delta = q_\epsilon \left(A^*A\right)A^*y^\delta
	\]
	with a filter function $q_\epsilon$ (see e.g. \cite{EHN1996,BHMR2007,MP2006,w18,bhr18,lw20}), or variational ones of the form
	\[
	u_\epsilon^\delta \in \argmin_{u \in \X} \left[ \left\Vert Au \right\Vert_Y^2 - 2 \left\langle Au, y^\delta\right\rangle_{\Y \times \Y^*} + \epsilon R \left(u\right)\right]
	\]
	where $R : X \to \left(-\infty, \infty\right]$ is a proper convex functional (see e.g. \cite{BHM2007,w12,hw17,wh20}). For a better understanding, note that $y \mapsto \left\Vert y \right\Vert_{\Y}^2 - 2 \left\langle y, y^\delta\right\rangle_{\Y \times \Y^*}$ is an infinite-dimensional version of the Gaussian negative log-likelihood functional. Note that both methods rely on a so-called regularization parameter $\epsilon > 0$, which has to be chosen appropriately.
	
	Given the datum $y^\delta$ in \eqref{eq_LInP}, many of the aforementioned methods perform optimal in the classical minimax sense, i.e. for $u^\dag$ in a prescribed smoothness class, the obtained convergence rate of the mean-squared error (MSE) $\E \| u_\epsilon^\delta - u^\dag\|^2 : = \E\left[ \left\Vert u_\epsilon^\delta - u^\dag\right\Vert_{\X}^2 \right]$ agrees with the best possible one under all estimators if $\epsilon >0$ is chosen appropriately. However, in many practical applications the single datum in \eqref{eq_LInP} arises from averaging several sequential (and independent) observations $y_1, ..., y_N$ according to the model
	\begin{equation}\label{eq:seq}
	y_i = A u^\dag + \eta_i, \qquad 1 \leq i \leq N,\quad \eta_i \sim \N_\Y\left(0, \Sigma_i\right)
	\end{equation}
	which yields
	\begin{equation}\label{eq:averaging}
	y^\delta := \frac{1}{N} \sum_{i=1}^N y_i = A u^\dag + \frac{1}{N} \sum_{i=1}^N \eta_i, \qquad \frac{1}{N} \sum_{i=1}^N \eta_i \sim\N_\Y\left(0, \frac1{N^2} \sum_{i=1}^N\Sigma_i\right)
	\end{equation}
	by independence, see also \cite{hjp20}. If the sequential noise contributions $\eta_i$ have identical covariance operators $\Sigma_i \equiv \Sigma$, then \eqref{eq:averaging} yields the original model \eqref{eq_LInP} with $\delta = 1/\sqrt{N}$. From this point of view, it might be advantageous to work with the sequence \eqref{eq:seq} of problems instead of the single problem \eqref{eq_LInP}. Note that \eqref{eq:seq} can also be interpreted as an (artificial) dynamical system
	\begin{subequations}\label{eq_Artidynamic}
		\begin{align}
		u_n & = u_{n-1} \label{eq_Artidynamica}\\
		y_n & = A u_n + \eta_n \label{eq_Artidynamicb}
		\end{align}
	\end{subequations}
	with $u_0 = u^\dag$ on a finite time horizon $n \in \left\{1,2,\ldots N\right\}$, similar to the one considered in \cite{ILS2013}.
	
	\subsection{Data assimilation as regularization}
	
	The artificial dynamic system \eqref{eq_Artidynamic} then allows us to apply classical data assimilation methods for the recovery of $u^\dag$, e.g. the Kalman filter and 3DVAR, which lead to a solution of the original inverse problem \eqref{eq_LInP} in form of a posterior Gaussian distribution. This also offers a connection to Bayesian inverse problems, see e.g. the seminal work \cite{stuart2010} or \cite{DLC2018,chkp19} for recent developments.
	
	For the sake of completeness, we briefly describe these approaches here. The Kalman filter yields the posterior Gaussian distribution $\N_{\X}\left(m_n,C_n\right)$ where
	\begin{subequations}\label{eq_KalmanDis}
		\begin{align}
		\label{eq:Kgain}K_{n} &=\tri{C}{n-1}{}A^*\left(A\tri{C}{n-1}{}A^*+\Sigma\right)^{-1}\\
		\label{eq:mean}\tri{m}{n}{}&=\tri{m}{n-1}{}+K_{n}(y_{n}-A\tri{m}{n-1}{})  \\
		\label{eq:covariance}\tri{C}{n}{}&= (I-K_{n}A)\tri{C}{n-1}{},
		\end{align}
	\end{subequations}
	with an initial (prior) distribution $\N_\X \left(m_0,C_0\right)$. Note that $\N_\X \left(m_0,C_0\right)$ is a tight probability if and only if the operator $C_0$ is of trace class, and this property is inherited by the posterior distribution. In \eqref{eq_KalmanDis}, $K_n$ is called the Kalman gain, $m_n$ is the posterior mean and $C_n$ is the posterior covariance. The well-known 3DVAR filter is obtained by fixing the posterior covariance, i.e. setting $K_{n}\equiv \mathcal{K}$, which yields the posterior Gaussian distribution $\N_\X\left(\zeta_n, \mathcal{C}\right)$ with
	\begin{subequations}\label{eq_3DVARN}
		\begin{align}
		\label{eq:3DVARKgain} K_{n}& \equiv  \mathcal{K} :=C_0 A^*\left(A C_0 A^*+ \Sigma\right)^{-1} \\
		\label{eq:3DVARmean} \tri{\zeta}{n}{} &= \tri{\zeta}{n-1}{}+\mathcal{K}(y_n-A\tri{\zeta}{n-1}{}), \\
		\label{eq:3DVARcovariance} \mathcal{C} &\equiv (I-\mathcal{K}A)C_0.
		\end{align}
	\end{subequations}
	Note that the computational effort for \eqref{eq_3DVARN} is considerably lower than for \eqref{eq_KalmanDis}, as the covariance operator $\mathcal C$ does not have to be updated in each iteration. Error bounds for both methods in the above setting have been investigated in \cite{ILLS2017} where a logarithmic difference between them is obtained.
	
	We shall mention that the artificial dynamic system (\ref{eq_Artidynamic}) has the further advantage that a nonlinear inverse problem can be solved by the ensemble Kalman filter (EnKF) without deriving the Fr\'{e}chet derivative of the forward operator, c.f. \cite{ILS2013}.
	We further refer to \cite{srs17} for estimation of parameters in dynamical systems, and to \cite{SS2017,bsww19} for recent error bounds of the EnKF.
	
	\subsection{Towards a continuous analog}\label{se13}
	
	In many applications arising from astrophysics to biomedical image analysis, the introduction of time-steps in \eqref{eq_Artidynamic} is artificial in the sense that the actual measurement process is rather discrete but continuous in time (see e.g. \cite{hw16,msw20} for mathematical reviews on the corresponding models). Therefore it is a natural question to ask for a continuous analog of \eqref{eq_Artidynamic} and the corresponding assimilation methods \eqref{eq_KalmanDis} and \eqref{eq_3DVARN}. It has already been pointed out in \cite[Ch.6]{LSZ2015} that the discrete system \eqref{eq_Artidynamic} can heuristically be transferred to a continuous one, but however our derivation here will be slightly different.
	Let us start by interpreting discrete state variables $u_n$ as equidistant (approximate) samples of a random process $u$ in the time interval $\left[0,T\right]$ (note that the ending time $T$ will in principle have the same meaning as $N$ before). Similarly we introduce (weak) random variables $z_1, ..., z_n$ as equidistant (approximate) samples of a random process $z$ to be observed such that $y_n = \left(\frac{z_{n}-z_{n-1}}{\tau}\right)$ with a time step $\tau>0$. If now $\tau \to 0$, then the number of observations within $\left[0,T\right]$ increases, and hence to obtain a meaningful limit, the covariances of the noise contributions in \eqref{eq_Artidynamicb} have to increase proportional to $\tau^{-1}$ as well. This also reflects the physical fact that an increased measurement frequency typically leads to a worse signal-to-noise ratio per observation. Thus assume that the white noise $\eta_n$ in \eqref{eq_Artidynamicb} have a covariance $\tau^{-1} \Sigma$. Then we can revise (\ref{eq_Artidynamicb}) into
	\begin{subequations}\label{eq_zn}
		\begin{align}
		z_{n} & = z_{n-1} + \tau A u_n + \sqrt{\tau\Sigma} \eta_{n},\quad n\in \mathbb{Z}^+, \\
		z_0 & = 0,
		\end{align}
	\end{subequations}
	with an i.i.d. sequence $\eta=\{\eta_n\}_{n\in \left\{1,... ,T/\tau\right\}}$ obeying $\eta_1 \sim \N_\Y\left(0,I\right)$. Here and in what follows, $I$ denotes the identity operator. Now it can readily be seen that $z_n$ in \eqref{eq_zn} is just given as the Euler-Maruyama approximation with time step $\tau$ of the continuous process $z$ in the SDE
	\begin{subequations}\label{eq_consys}
		\begin{align}
		& \mathrm{d}u = 0, \quad u(0) = u^\dag; \label{eq_consysa}\\
		& \mathrm{d}z = Au\mathrm{d}t + \sqrt{\Sigma}\mathrm{d}W, \quad z(0)=0, \label{eq_consysb}
		\end{align}
	\end{subequations}
	on $\left[0, T\right]$ where $W$ is the standard Wiener process. As a consistency check, we note that
	\[
	y^\delta := \frac{1}{T} z\left(T\right)= A u^\dag + \frac1T \sqrt{\Sigma} \left(W\left(T\right) - W\left(0\right)\right), \qquad \sqrt{\Sigma} \left(W\left(T\right) - W\left(0\right)\right) \sim \N_{\Y} \left(0, T \Sigma\right).
	\]
	Hence, the ending point of the observable process $z$ carries the same information as the data observed in the original inverse problem \eqref{eq_LInP} with $\delta = 1/\sqrt{T}$.
	
	\subsection{Continuous data assimilation approaches as regularization methods and aims of this paper}\label{se14}
	
	The above reformulation, in particular the continuous system \eqref{eq_consys}, allows us to implement the classic Kalman-Bucy filter in data assimilation and to derive SDEs involving the estimator $m\left(t\right)$ for the state variable $u(t)$, which is assumed to be time-independent according to \eqref{eq_consysa}. More precisely, referring to \cite[Ch.6]{LSZ2015}, we can obtain the following system
	\begin{subequations}\label{eq_KBfilter}
		\begin{align}
		\mathrm{d}m &= CA^*\Sigma^{-1}(\mathrm{d}z-Am\mathrm{d}t), \quad m(0)=m_0; \label{eq_KBfiltera}  \\
		\mathrm{d}C &= -CA^*\Sigma^{-1}AC\mathrm{d}t, \quad C(0)=C_0. \label{eq_KBfilterb}
		\end{align}
	\end{subequations}
	Concerning $C_0$, the same comments as after \eqref{eq_KalmanDis} apply. It is immediately clear that the posterior distribution of $u|z$ is Gaussian with the mean $m(t)$ and covariance $C(t)$. Note that - as $z$ is observable on $\left[0,T\right]$ only - the Kalman-Bucy filter and its mean function $m$ are well-defined on $\left[0,T\right]$ only. In the limit $T \to \infty$ we expect convergence $m\left(T\right) \to u^\dag$, which will be investigated in Section \ref{se3}.
	
	To obtain the posterior distribution, one need to firstly solve the Riccati equation (\ref{eq_KBfilterb}) for the posterior covariance $C(t)$ and substitute it into (\ref{eq_KBfiltera}) to further derive the posterior mean $m(t)$. In general, the Riccati equation can not be solved explicitly. Nevertheless, because of the stationary state equation (\ref{eq_consysa}) we are able to write down the solution of (\ref{eq_KBfilter}). Actually, without loss of generality, we assume that $C(t)$ is positive definite for any finite time $t>0$. Then, the inverse of $C(t)$, denoted by $C^{-1}(t)$, is well-defined at any finite time $t$ which yields
	\begin{align*}
	0=\mathrm{d}[C(t)C^{-1}(t)] =[\mathrm{d}C(t)]C^{-1}(t)+C(t)[\mathrm{d}C^{-1}(t)]\notag.
	\end{align*}
	Hence by substituting (\ref{eq_KBfilterb}), we obtain
	\begin{align*}
	\mathrm{d}C^{-1}(t) = -C^{-1}(t)[\mathrm{d}C(t)]C^{-1}(t) = A^*\Sigma^{-1}A \mathrm{d}t,
	\end{align*}
	and it is straightforward to derive $C^{-1}(t) = C_0^{-1} + tA^*\Sigma^{-1}A$ since $A^*\Sigma^{-1}A$ is time-independent. Equivalently we can write
	\begin{align}\label{eq_KFcov}
	C(t) = ( C_0^{-1} + tA^*\Sigma^{-1}A)^{-1}, \quad t>0,
	\end{align}
	and insert it into (\ref{eq_KBfiltera}) to obtain the following initial value problem
	\begin{align}\label{eq_KFm}
	\mathrm{d}m = (C^{-1}_0 + tA^*\Sigma^{-1}A)^{-1} A^*\Sigma^{-1}(\mathrm{d}z-A m \mathrm{d}t), \quad m(0) = m_0
	\end{align}
	which will be called the {\it non-stationary \textbf{A}symptotical \textbf{R}egularization \textbf{M}ethod (non-stationary ARM)} .
	
	On the other hand, we can also consider some approximate Gaussian (continuous) filter such as the 3DVAR by fixing the posterior covariance in (\ref{eq_KBfilter}). Then the posterior mean and covariance, denoted by $\zeta(t)$ and $\mathcal{C}(t)$, is formally obtained by
	\begin{subequations}\label{eq_3dvar}
		\begin{align}
		\mathrm{d}\zeta &= \mathcal{C}A^*\Sigma^{-1}(\mathrm{d}z-A \zeta\mathrm{d}t), \quad \zeta(0)=m_0, \label{eq_3dvara}\\
		\mathrm{d}\mathcal{C} &= 0, \quad \mathcal{C}(0)=C_0 \label{eq_3dvarb}
		\end{align}
	\end{subequations}
	which is called the {\it stationary \textbf{A}symptotical \textbf{R}egularization \textbf{M}ethod (stationary ARM)}.
	
	The aim of this paper is to derive error bounds for the asymptotical regularization methods \eqref{eq_KFm} (or \eqref{eq_KBfilter}), \eqref{eq_3dvar} under standard assumptions, and to compare these results with classical regularization methods for the original inverse problem \eqref{eq_LInP} such as Tikhonov and Showalter regularization. Noticing that the method (\ref{eq_3dvar}), which does not update the posterior covariance, is computationally more efficient than the method (\ref{eq_KFm}), the quantitative difference between them will also be revealed. One essential point is that both \eqref{eq_KFm} and \eqref{eq_3dvar} allow for an \textit{online}-type reconstruction of the unknown quantity $u^\dag$, whereas classical concepts from regularization theory can only be applied after gathering and averaging all data. Hence, it is an interesting question if this advantage comes for free (at least asymptotically in the sense that the rates of convergence agree as $T\to \infty$ and $\delta \to 0$), or if there is a price to pay for these immediate availability of reconstructions.
	
	Note that we also extend the study of linear statistical inverse problems to a continuous form, which yields a Wiener process and is novel in error bound analysis of the asymptotical regularization. It is worth to emphasize that inverse problems of differential equations with Wiener processes have attracted much attention recently and we mention \cite{BCLZ2014,DH2014,KP2018}.
	
	The outline of this study is as follows. In Section \ref{se2} we present our standing assumptions and provide a brief discussion of necessary techniques from regularization theory and stochastic calculus. The main error bounds are derived in Section \ref{se3} where the quantitative difference between both methods are presented. In Section \ref{se4}, numerical examples confirm the theoretical results and Section \ref{se5} ends the study with a discussion and possible future extensions.

	\section{Assumptions and necessary concepts}\label{se2}
	In this section we state our main assumptions and provide necessary concepts for further investigation.
	
	\subsection{Assumptions}
	
	To obtain error estimates of the non-stationary and stationary ARM \eqref{eq_KFm} and \eqref{eq_3dvar} derived from the Kalman-Bucy filter and 3DVAR, we need to pose some standard assumptions. In particular, we shall measure the smoothness of the exact solution $u^{\dag}$ related to the forward operator in certain sense by {\it source conditions.} An extended discussion on related topics can be found in \cite{EHN1996,MP2003,LP2013} and the references therein.
	
	Before we proceed further, the following assumption on the noise covariance operator and the initial covariance operator is posed:
	\begin{assp}\label{assp_1}
		The noise covariance operator $\Sigma$ is self-adjoint and positive definite. The initial covariance is chosen as $C_0 = \alpha^{-1} \Omega$ with a tuning parameter $\alpha > 0$ and a self-adjoint, positive definite trace class operator $\Omega$.
	\end{assp}
	Note that, as already mentioned below \eqref{eq_KalmanDis}, under Assumption \ref{assp_1} both means $m$ and $\zeta$ of \eqref{eq_KFm} and \eqref{eq_3dvar} are tight probabilities in the sense that their posterior covariances are of trace class as well and hence $\mathbb P \left[m \left(t\right) \in \X\right] = \mathbb P \left[ \zeta \left(t\right) \in \X\right] = 1$ for all $t \in \left[0,1\right]$.
	
	The tuning parameter $\alpha$ will later be chosen depending on $t$ (or the ending time $T$) to obtain convergence (and also an optimal convergence behavior) of $m$ and $\zeta$, respectively.
	
	Since there appear several operators $A$, $\Sigma$, $\Omega$ in both methods \eqref{eq_KFm} and \eqref{eq_3dvar},  similar to the reformulation in \cite{LLM}, we pre-whiten the original artificial dynamic \eqref{eq_Artidynamic} by multiplying with $\Sigma^{-1/2}$ on both sides and assume the following:
	
	\begin{assp}\label{assp_2}		
		\begin{enumerate}
			\item $\mathcal{R}(A) \subset \mathcal{D}(\Sigma^{-1/2})$. Denote $K :=\Sigma^{-1/2} A$ be the modified forward operator.
			\item The trace class prior covariance operator $\Omega$ is chosen as a power of $K^*K$ such that there exists a constant $p>0$ and $\Omega = (K^*K)^p$.
			\item Without loss of generality, we assume that $\|K\|\leq 1$.
		\end{enumerate}
	\end{assp}
	
	Note that Item 1 in Assumption \ref{assp_2} is necessary to allow for pre-whitening. In a particular case, one may choose $\Sigma=I$ representing the white noise and consequently $K=A$. Item 2 ensures that $\Omega$ and functions of $K^*K$ commute. Item 3 is more technical for the proof in Section \ref{se3} and can be guaranteed by re-scaling the norms in $\X$ and $\Y$.
	
	Under Assumption \ref{assp_2} we thus introduce a new operator $B:=\Sigma^{-1/2} A \Omega^{1/2}=K\Omega^{1/2}$ such that
	\begin{align}\label{eq_Boperator}
	B^*B = (K^* K)^{p+1}, \quad \text{and} \quad \Omega = (B^*B)^{\frac{p}{p+1}}.
	\end{align}
	
	Another assumption concerns the smoothness of the unknown solution $u^{\dag}$ which is usually described by the source condition. Here we focus on the spectral source conditions as considered e.g. in \cite{MP2003,LP2013}. Therefore recall, that an \textit{index function} is a non-decreasing and continuous function $\varphi : \left(0,\infty\right)\to \left(0,\infty\right)$ with $\lim_{\lambda \searrow 0} \varphi\left(\lambda\right) = 0$. Following a similar way as in \cite{FS2012,LLM}  the general source condition is introduced upon the modified forward operator $K$ and presented below.
	\begin{assp}\label{assp_3}
		We assume that there exists an index function $\varphi$ such that
		\begin{align*}
		m_0-u^{\dag} \in \mathcal{A}_{\varphi}:= \{x, x=\varphi(K^*K)v, \,\, \|v \| \leq 1\}.
		\end{align*}
	\end{assp}
	
	Note that this assumption is suitable for both the stationary and the non-stationary ARM, as we have assumed $\zeta \left(0\right) = m\left(0\right)= m_0$.
	
	The most common example of an index function $\varphi$ is
	\begin{align*}
	\varphi(\lambda) = \lambda^{\nu}, \quad \nu>0,
	\end{align*}
	in which case the corresponding smoothness assumption is called a \textit{H\"older source condition}. It is well-known that such assumptions are reasonable in moderately ill-posed problems, c.f. \cite{EHN1996}. On the other hand, for the exponentially ill-posed problems it is reasonable to consider a \textit{logarithmic source conditions} where
	\begin{align*}
	\varphi(\lambda) = \left(-\ln \lambda\right)^{-p}, \quad p>0,
	\end{align*}
	c.f. \cite{Hohage1997}.
	
	\subsection{Tools from regularization theory}
	
	In our analysis, several concepts from regularization theory will turn out useful. Therefore we recall the notaion of a qualification and the residual function, c.f. \cite{EHN1996,LP2013}.
	\begin{myDef}
		A family $\left(q_\epsilon\right)_{\epsilon > 0}$ of measurable functions
		\begin{align*}
		q_{\epsilon}(\lambda): [0,\|B^*B\|] \rightarrow \mathbb{R}
		\end{align*}
		is called a \textit{regularization} if
		\begin{align*}
		\sup_{0<\lambda\leq \|B^*B\|} |q_{\epsilon}(\lambda)| & \leq \frac{C_{-1}}{\epsilon} \qquad\text{for all}\quad  \epsilon > 0
		\end{align*}
		with a positive constant $C_{-1}$, and if its \textit{residual function} $r_{\epsilon}(\lambda) := 1-q_{\epsilon}(\lambda)\lambda$ satisfies
		\begin{align*}
		\sup_{0<\lambda\leq \|B^*B\|} |r_{\epsilon}(\lambda)| & \leq C_0\qquad\text{for all}\quad  \epsilon > 0
		\end{align*}
		with a positive constant $C_0$.
		
		The index $\nu_0>0$ is called the {\it qualification} of $\left(q_\epsilon\right)_{\epsilon > 0}$ if there exists a constant $c_\nu$ such that
		\begin{align*}
		\sup_{\lambda\in (0, \|B^*B\|]} |\lambda^\nu r_{\epsilon}(\lambda)| \leq c_{\nu} \epsilon^{\nu} \qquad\text{for all}\quad \epsilon > 0  \quad \text{and} \quad  0 \leq \nu \leq \nu_0.
		\end{align*}
	\end{myDef}
	
	We provide two examples of regularization methods, which will be useful in the following:
	\begin{example}
		\begin{description}
			\item[Tikhonov regularization]
			For $0 < \lambda,\epsilon \leq \|B^*B\|$, we let
			\begin{align*}
			q_{1,\epsilon}(\lambda) & := \frac{1}{\lambda+\epsilon},\\
			r_{1,\epsilon}(\lambda) & = \frac{\epsilon}{\lambda+\epsilon}.
			\end{align*}
			Referring to \cite{EHN1996, LP2013} we have
			\begin{align}\label{eq_resi_error}
			\sup_{\lambda} r_{1,\epsilon}(\lambda) \lambda^{\nu} \leq \epsilon^{\nu},\quad \epsilon>0, \quad 0\leq \nu\leq 1, \quad \lambda \in (0,\|B^*B\|].
			\end{align}
			
			\item[Showalter regularization]
			For $0 < \lambda \leq \|B^*B\|$, $0< \epsilon$, we let
			\begin{align*}
			q_{2,\epsilon}(\lambda) & := \frac{1-e^{-\frac{\lambda}{\epsilon}}}{\lambda},\\
			r_{2,\epsilon}(\lambda) & = e^{-\frac{\lambda}{\epsilon}}.
			\end{align*}
			Referring to \cite{T1994}, we have
			\begin{align}\label{eq_resi2_error}
			\sup_{0\leq \lambda\leq 1}r_{2,\epsilon}(\lambda) \lambda^{\mu} \leq c \left(1+\frac{1}{\epsilon}\right)^{-\mu} \leq c \epsilon^{\mu},\quad \epsilon>0, \quad  \mu \geq 0, \quad \lambda\in (0,\|B^*B\|],\quad \|B^*B\| \leq 1.
			\end{align}
			The constant $c$ in (\ref{eq_resi2_error}) is $c = \max\{\mu^{\mu},1\}$.
		\end{description}
		
	\end{example}
	
	We shall emphasize that the inequalities (\ref{eq_resi_error}) and (\ref{eq_resi2_error}) shed light on the qualification of both residual functions, where $r_{1,\epsilon}(\cdot)$ yields a qualification of $\nu_0=1$ for Tikhonov regularization and $r_{2,\epsilon}(\cdot)$ yields a qualification of $\nu_0=\infty$ for Showalter regularization.
	
	In the end of current subsection, we introduce the effective dimension $\mathcal{N}(\epsilon)$ of the operator $B$ defined by
	\begin{align}\label{eq_EffDim}
	\mathcal{N}(\epsilon)=\mathcal{N}_B(\epsilon):= \tr\left((\epsilon I+B^*B)^{-1}B^*B\right), \epsilon>0.
	\end{align}
	The value of $\mathcal{N}(\epsilon)$ depends on the singular values of $B$ and, in the infinite dimensional setting, it yields H\"{o}lder type or logarithmic type asymptotics with respect to the power-type or exponential decay of singular values of $B$. We refer to \cite[Lem. 2.2]{LLM} for some properties of the effective dimension, which will be recalled in bounding the error estimate below.
	
	\subsection{Necessary concepts of stochastic calculus}\label{sec:stochasticinte}
	Note that in both methods (\ref{eq_KFm}) and (\ref{eq_3dvar}), there appear some stochastic integrals with respect to a Wiener process. To obtain corresponding error bounds, we will make use of some techniques from stochastic calculus to be presented in the current subsection. Most concepts can be found in \cite[Chap.2]{gm11} and we collect them here for sake of completeness.
	
	Recall that $\X$ and $\Y$ are separable Hilbert spaces, denote $Q$ be a self-adjoint positive semi-definite trace class operator on $\Y$, and by $\lambda_j>0, f_j$, $j=1,2,\ldots$ all its eigenvalues and eigenvectors. Then we can define the separable Hilbert space $\Y_Q= Q^{1/2} \Y$ equipped with the scalar product
	\begin{align*}
	\langle w, v \rangle_{\Y_Q} = \sum_{j=1}^{\infty}\frac{1}{\lambda_j} \langle w,f_j \rangle_\Y  \langle v,f_j \rangle_\Y.
	\end{align*}
	For a sequence $\{\omega_j(t)\}$, $j=1,2,\ldots$ of independent Brownian motions, the {\it $\Y$-valued $Q$-Wiener process $\mathcal{W}(t)$} is defined by
	\begin{align*}
	\mathcal{W}(t) = \sum_{j=1}^{\infty} \lambda^{1/2}_j \omega_j(t) f_j.
	\end{align*}
	
	Denote by $\mathcal{L}_2(\Y_Q,\X)$ the space of Hilbert-Schmidt operators from $\Y_Q$ to $\X$. If $\{e_j\}_{j=1}^{\infty}$ is a complete orthonormal system in $\X$, then the Hilbert-Schmidt norm of an operator $L \in \mathcal{L}_2(\Y_Q,\X)$ is given by
	\begin{align*}
	\|L\|^2_{\mathcal{L}_2(\Y_Q,\X)} & = \sum_{j,i=1}^{\infty} \langle L(\lambda^{1/2}_j f_j),e_i\rangle_\X^2 = \sum_{j,i=1}^{\infty} \langle L Q^{1/2} f_j,e_i\rangle_\X^2 \\
	& = \|LQ^{1/2}\|^2_{\mathcal{L}_2(\Y,\X)} = \tr\left((LQ^{1/2}) (LQ^{1/2})^*\right)\\
	& = \tr\left(LQL^*\right).
	\end{align*}
	
	Let now $\Lambda_2(\Y_Q,\X)$ be the class of $\mathcal{L}_2(\Y_Q,\X)$-valued processes that satisfy the condition
	\begin{align*}
	\E \int_0^T \|\Phi(s)\|^2_{\mathcal{L}_2(\Y_Q,\X)} \mathrm{d}s < \infty.
	\end{align*}
	One can verify that $\Lambda_2(\Y_Q,\X)$ is a Hilbert space equipped with the norm
	\begin{align*}
	\|\Phi\|_{\Lambda_2(\Y_Q,\X)} = \left(\E \int_0^T \|\Phi(s)\|^2_{\mathcal{L}_2(\Y_Q,\X)} \mathrm{d}s \right)^{1/2}.
	\end{align*}
	
	For $\Phi \in \mathcal{L}_2(\Y_Q,\X)$, the stochastic integral $\int_0^t \Phi(s) \mathrm{d}\mathcal{W}(s)$, $0 \leq t \leq T$, can be defined just as in the finite dimensional case based on elementary processes and continuous extension, see \cite[Sec. 2.2]{gm11} for details. The following theorem in \cite{gm11}, which is the It\^{o}-isometry in the infinite-dimensional setting, is important and forms the main tool to handle the stochastic integrals in current work.
	\begin{thm}\cite[see Theorem 2.3]{gm11}\label{thm:si}
		The stochastic integral $\Phi \rightarrow \int_0^t \Phi(s) \mathrm{d}\mathcal{W}(s)$ with respect to a $\Y$-valued $Q$-Wiener process $\mathcal{W}(s)$ satisfies
		\begin{align*}
		\E \left\|\int_0^t \Phi(s) \mathrm{d}\mathcal{W}(s)\right\|^2_\X = \E \int_0^t \|\Phi(s)\|^2_{\mathcal{L}_2(\Y_Q,\X)} \mathrm{d}s < \infty
		\end{align*}
		for $t\in [0,T]$.
	\end{thm}
	
	\section{Bounds for mean squared error}\label{se3}
	In this section we present our main results consisting of error bounds for both methods \eqref{eq_KFm} and \eqref{eq_3dvar} on the MSE
	\begin{align}\label{eq_mse}
	\E\|m(t)-u^\dag\|^2, \qquad \E\|\zeta(t)-u^\dag\|^2
	\end{align}
	where $m(t)$ (or $\zeta(t)$) is the posterior mean derived by the non-stationary (or stationary) ARM in Subsection \ref{se14}, respectively. We will derive bounds for both quantities whenever $0 \leq t \leq T$, even though $t = T$ is - in view of \eqref{eq_LInP} - the most interesting case as it contains full data in the whole time interval $[0,T]$. The MSE estimates will be carried out by the classic bias-variance decomposition
	\begin{align}\label{eq_bvdecomp}
	\E\|m(t)-u^\dag\|^2  = \|\E m(t)-u^{\dag}\|^2 + \E \|m(t) - \E m(t)\|^2,
	\end{align}
	and analogously for $\zeta\left(t\right)$. In the right-hand side of above equality \eqref{eq_bvdecomp}, we call $\|\E m(t)-u^{\dag}\|^2 $ be the bias term and $\E \|m(t) - \E m(t)\|^2$ be the variance term.
	
	\subsection{Non-stationary ARM} 
	To bound the MSE, we first derive an explicit formula for the error between the posterior mean and the unknown exact solution. The non-stationary ARM \eqref{eq_KBfilter} has an updating covariance operator $C(t)$ which varies when the time variable $t$ increases. The calculation in Subsection \ref{se14} allows us to write down the covariance operator $C(t)$ and derive the equivalent form (\ref{eq_KFm}) of the posterior mean $m(t)$.
	Noticing the fact that $\mathrm{d}u=0$ and $\mathrm{d}z = Au\mathrm{d}t +\sqrt{\Sigma} \mathrm{d}W$, we could reform (\ref{eq_KFm}) into
	\begin{align}\label{eq_KFmerror}
	\left\{
	\begin{array}{l}
	\mathrm{d}(u-m) = -(C^{-1}_0 + tA^*\Sigma^{-1}A)^{-1} A^*\Sigma^{-1}A (u-m) \mathrm{d}t - (C^{-1}_0 + tA^*\Sigma^{-1}A)^{-1} A^*\Sigma^{-1/2} \mathrm{d}W,  \\
	u(0)-m(0) = u^{\dag}-m_0.
	\end{array}
	\right.
	\end{align}
	Here we denote $u(t)\equiv u^{\dag}$ be the exact solution which is deterministic and stationary with respect to the time variable $t$.
	
	Solving the above initial value problem (\ref{eq_KFmerror}), we obtain the solution $(u -m)(t)$ by
	\begin{align}\label{eq_KFsolution}
	(u-m)(t)  = &  e^{-\int_0^t (C_0^{-1} + s A^* \Sigma^{-1} A)^{-1} A^* \Sigma^{-1} A \mathrm{d}s } (u-m)(0) \nonumber \\
	& \quad - \int_0^t e^{-\int_s^t  (C_0^{-1} + \tau A^* \Sigma^{-1} A)^{-1} A^* \Sigma^{-1} A \mathrm{d}\tau }  (C_0^{-1} + s A^* \Sigma^{-1} A)^{-1} A^* \Sigma^{-1/2} \mathrm{d} W(s).
	\end{align}
	By elementary operator calculations, cf. \cite{B1997}, we find
	\begin{align}\label{eq_expintegral}	
	e^{-\int_s^t  (C_0^{-1} + \tau A^* \Sigma^{-1} A)^{-1} A^* \Sigma^{-1} A \mathrm{d}\tau } = (C_0^{-1} + t A^* \Sigma^{-1} A)^{-1}(C_0^{-1} + s A^* \Sigma^{-1} A).
	\end{align}
	Then we insert the exact solution $u(t) = u^{\dag}$, the initial mean $m(0)=m_0$ and rewrite \eqref{eq_KFsolution} into
	\begin{align}\label{eq_KBm}
	\begin{split}
	u^{\dag} - m(t)& = (C_0^{-1} + t A^* \Sigma^{-1} A)^{-1} C_0^{-1} (u^{\dag}-m_0)  \\
	&\qquad  - \int_0^t e^{-\int_s^t  (C_0^{-1} + \tau A^* \Sigma^{-1} A)^{-1} A^* \Sigma^{-1} A \mathrm{d}\tau }  (C_0^{-1} + s A^* \Sigma^{-1} A)^{-1} A^* \Sigma^{-1/2} \mathrm{d} W(s).
	\end{split}
	\end{align}
	Hence, the bias-variance decomposition involves the two terms
	\begin{subequations}\label{eq_KFdecom}
		\begin{align}
		u^{\dag} - \E m(t) & = (C_0^{-1} + t A^* \Sigma^{-1} A)^{-1} C_0^{-1} (u^{\dag}-m_0),  \label{eq_KFdecoma}\\
		\E m(t) - m(t) &  = - \int_0^t e^{-\int_s^t  (C_0^{-1} + \tau A^* \Sigma^{-1} A)^{-1} A^* \Sigma^{-1} A \mathrm{d}\tau }  (C_0^{-1} + s A^* \Sigma^{-1} A)^{-1} A^* \Sigma^{-1/2} \mathrm{d} W(s),  \label{eq_KFdecomb}
		\end{align}
	\end{subequations}
	and we bound each term separately. Note that \eqref{eq_KFdecomb} consists of an infinite dimensional stochastic integral, which has to be treated with some care. 	
	
	As the bias term is deterministic, we bound it below by standard techniques in regularization theory.
	\begin{prop}\label{prop_1}
		Let Assumptions \ref{assp_1}-\ref{assp_3} hold, then the non-stationary ARM yields error bounds of the bias term
		\begin{enumerate}
			\item If the function $\lambda \mapsto \varphi(\lambda)/\lambda^{p+1}$ is non-increasing, then
			\begin{align*}
			\|\E m(t)-u^{\dag}\|^2 \leq \varphi^2\left(\left(\frac{\alpha}{t}\right)^{\frac{1}{p+1}}\right)\qquad\text{for all}\qquad 0 \leq t \leq T.
			\end{align*}
			\item If there is a constant $c<\infty$ with $\varphi(\lambda) \leq c \lambda^{p+1}$ as $\lambda\rightarrow 0$, then
			\begin{align*}
			\|\E m(t)-u^{\dag}\|^2 \leq c \left(\frac{\alpha}{t}\right)^2 \qquad\text{for all}\qquad 0 \leq t \leq T.
			\end{align*}
		\end{enumerate}
		
	\end{prop}
	\begin{proof}
		Using Assumptions \ref{assp_1} and \ref{assp_2} we rewrite (\ref{eq_KFdecoma}) into
		\begin{align*}
		u^{\dag} - \E m(t) & = \Omega^{1/2} \frac{\alpha}{t} \left(\frac{\alpha}{t} + B^* B\right)^{-1} \Omega^{-1/2}(u^{\dag}-m_0) \nonumber \\
		& = \frac{\alpha}{t} \left(\frac{\alpha}{t} + B^* B\right)^{-1}(u^{\dag}-m_0)
		\end{align*}
		where the latter equality follows after the commuting property between $\Omega^{1/2}$ and $B^*B$ under Assumption \ref{assp_2}.
		We thus obtain, by using Assumption \ref{assp_3},
		\begin{align*}
		\|\E m(t)-u^{\dag}\|^2 = \|r_{1,\frac{\alpha}{t}}(B^*B) \varphi(K^*K)\|^2.
		\end{align*}
		A direct call of \cite[Lemma 3.1]{LLM} or implementation of (\ref{eq_resi_error}) then yields the results by viewing $\alpha/t$ as the regularization parameter.
	\end{proof}
	
	To treat the variance term, we need to investigate the stochastic integral in \eqref{eq_KFdecomb} carefully and provide its bound below.
	
	\begin{prop}\label{prop_2}
		Let Assumptions \ref{assp_1}-\ref{assp_2} hold, then the non-stationary ARM yields a bound of the variance term
		\begin{align*}
		\E\|\E m(t) - m(t)\|^2 & \leq  \min \left\{\alpha^{-1} \tr \left(\Omega\right), \alpha^{-\frac{1}{p+1}} t^{-\frac{p}{p+1}} \mathcal{N}\left(\frac{\alpha}{t}\right) \right\}\qquad\text{for all}\qquad 0 \leq t \leq T.
		\end{align*}
	\end{prop}
	\begin{proof}
		We first rewrite
		\begin{align*}
		& \E m(t) - m(t) \\
		& \quad  = - \int_0^t e^{-\int_s^t  (C_0^{-1} + \tau A^* \Sigma^{-1} A)^{-1} A^* \Sigma^{-1} A \mathrm{d}\tau }  (C_0^{-1} + s A^* \Sigma^{-1} A)^{-1} A^* \Sigma^{-1/2} (B B^*)^{-\frac{p}{2(p+1)}}(B B^*)^{\frac{p}{2(p+1)}} \mathrm{d} W(s) .
		\end{align*}
		In particular, we denote $\mathrm{d} \mathcal{W}(s): = (B B^*)^{\frac{p}{2(p+1)}} \mathrm{d} W(s)$ where $ \mathcal{W}(s)$ is a $Q$-Wiener process for $Q := (B B^*)^{\frac{p}{p+1}}$. As $B^*B$ and $B B^*$ have the same eigenvalues\footnote{If $u$ is an eigenfunction for $B^*B$ with eigenvalue $\lambda$, then $B B^*B u = B \left(B^*B\right) u = B \lambda u = \lambda Bu$, i.e. $Bu$ is an eigenfunction of $BB^*$ with the same eigenvalue and vice versa.}, the operator $Q$ is a positive definite self-adjoint trace class operator according to Assumption \ref{assp_1} and Item 2 of Assumption \ref{assp_2}.
		
		Meanwhile we define the following process
		\begin{align*}
		\Pi(s) : =   e^{-\int_s^t  (C_0^{-1} + \tau A^* \Sigma^{-1} A)^{-1} A^* \Sigma^{-1} A \mathrm{d}\tau }  (C_0^{-1} + s A^* \Sigma^{-1} A)^{-1} A^* \Sigma^{-1/2} B^* (B B^*)^{-\frac{p}{2(p+1)}}
		\end{align*}
		and will verify that the square of the Hilbert-Schmidt norm $ \|\Pi(s)\|^2_{\mathcal{L}_2(\Y_Q,\X)} = \tr \left(\Pi Q \Pi^*\right)$ is bounded. To this end, we recall \eqref{eq_expintegral}, Assumptions \ref{assp_1}, \ref{assp_2}  and rewrite the process by
		\begin{align*}
		\Pi(s) & = (C_0^{-1} + t A^* \Sigma^{-1} A)^{-1}   A^* \Sigma^{-1/2} (B B^*)^{-\frac{p}{2(p+1)}} \\
		& = \frac{1}{t} \Omega^{1/2}\left(\frac{\alpha}{t} + B^* B\right)^{-1} B^* (B B^*)^{-\frac{p}{2(p+1)}}.
		\end{align*}
		One then can prove that $\Pi(s)$ is a bounded process with respect to the variable $s$ for any fixed $\alpha, t \in (0,\infty)$ noticing that
		\begin{align*}
		\sup_{\lambda\in (0,\|B^*B\|]} \left|\frac{1}{t} \lambda^{\frac{p}{2(p+1)}}\left(\frac{\alpha}{t} + \lambda\right)^{-1}  \lambda^{-\frac{p}{2(p+1)}} \lambda^{1/2} \right| \leq \frac{1}{\sqrt{\alpha t}}.
		\end{align*}
		From Theorem \ref{thm:si} we verify that
		\begin{align*}
		\E\|\E m(t) - m(t)\|^2 = \E \left\|\int_0^t \Pi(s) \mathrm{d} \mathcal{W}(s) \right\|^2 = \E \int_0^t \|\Pi(s)\|^2_{\mathcal{L}_2(\Y_Q,\X)}   \mathrm{d}s < \infty.
		\end{align*}
		
		By using Assumptions \ref{assp_1}-\ref{assp_2} again, we end the proof by further implementing the cyclic property of the trace operator and (\ref{eq_resi_error}) to derive
		\begin{align*}
		\E\|\E m(t) - m(t)\|^2 & = \frac{1}{t}\tr\left(\Omega \left(\frac{\alpha}{t} + B^* B\right)^{-2} B^*B\right)  \nonumber \\
		& \leq \frac{1}{\alpha^2} t \left\| \frac{\alpha}{t} \left(\frac{\alpha}{t} + B^* B\right)^{-1}  \sqrt{B^* B}\right\|^2 \tr \left(\Omega\right) \\
		& \leq \alpha^{-1} \tr \left(\Omega\right)
		\end{align*}
		or
		\begin{align*}
		\E\|\E m(t) - m(t)\|^2 & = \frac{1}{t}\tr\left(\Omega \left(\frac{\alpha}{t} + B^* B\right)^{-2} B^*B\right)  \nonumber \\
		& \leq \frac{1}{t} \left\|\Omega \left(\frac{\alpha}{t} + B^* B\right)^{-1}\right\| \tr \left(\left(\frac{\alpha}{t} + B^* B\right)^{-1} B^*B\right) \\
		& \leq \frac{1}{\alpha} \sup_{\lambda\in (0,\|B^*B\|}\left|\lambda^{\frac{p}{p+1}} r_{1,\frac{\alpha}{t}}(\lambda)\right| \mathcal{N}\left(\frac{\alpha}{t}\right), \\
		& = \alpha^{-\frac{1}{p+1}} t^{-\frac{p}{p+1}} \mathcal{N}\left(\frac{\alpha}{t}\right)
		\end{align*}
		where the term $\tr \left(\left(\frac{\alpha}{t} + B^* B\right)^{-1} B^*B\right) $ is the effective dimension $\mathcal{N}\left(\frac{\alpha}{t}\right)$ of $B^*B$.
	\end{proof}
	
	\begin{rem}\label{rem_1}
		Note, that introducing $\mathrm{d} \mathcal{W}(t): = (B B^*)^{\frac{p}{2(p+1)}} \mathrm{d} W(t)$ can be considered as a pre-smoothing step, which transforms the white noise $W$ into $\Y$-valued noise $\mathcal W$. However, we do this only for the analysis of the variance term here (which avoids replacing the whole problem \eqref{eq_LInP} by a smoother but more ill-posed one), and this is furthermore only possible because we assumed the initial covariance $\Omega$ to be of trace class.
	\end{rem}
	
	\begin{rem}\label{rem_2}
		We shall mention that in both finite and infinite-dimensional settings, the variance term is asymptotically decaying when the time variable $t$ becomes large. Extended discussion is provided here.
		\begin{description}
			\item[Case 1. Finite dimensional setting:] Note that
			\begin{align*}
			\E\|\E m(t) - m(t)\|^2 & = \frac{1}{t}\tr\left(\Omega \left(\frac{\alpha}{t} + B^* B\right)^{-2} B^*B\right)  \nonumber \\
			& \leq \frac{1}{\alpha^2} t \left\| \frac{\alpha}{t} \left(\frac{\alpha}{t} + B^* B\right)^{-1}  \sqrt{B^* B \Omega}\right\|^2 \tr \left(I\right) \\
			& \leq d \alpha^{-\frac{1}{p+1}} t^{-\frac{p}{p+1}}
			\end{align*}
			where $d$ is the dimensionality of the state variable $u^{\dag}$ as well as the upper bound of $\mathcal{N}\left(\frac{\alpha}{t}\right)$.
			\item[Case 2. Moderately ill-posed operator:] Assume that the singular value $s_j^2$ of $B^* B$ decays in a polynomial manner, i.e. $s_j^2 \asymp j^{-2\theta}$ for some $\theta>0$, then (\ref{eq_Boperator}) and Item 2 of Assumption \ref{assp_2} yields
			\begin{align*}
			2\theta \frac{p}{p+1}> 1.
			\end{align*}
			At the same time, we recall the asymptotical behavior of the effective dimension in \cite[Page. 901]{LLM} such that $\mathcal{N}\left(\frac{\alpha}{t}\right) \asymp \left(\frac{\alpha}{t}\right)^{-\frac{1}{2\theta}}$ . Then we obtain
			\begin{align*}
			\E\|\E m(t) - m(t)\|^2 & \leq \alpha^{-\frac{1}{p+1}} t^{-\frac{p}{p+1}} \mathcal{N}\left(\frac{\alpha}{t}\right) \\
			& \asymp \alpha^{-\frac{1}{p+1}-\frac{1}{2\theta}} t^{-\frac{p}{p+1} + \frac{1}{2\theta}}.
			\end{align*}
			Noticing that $\frac{p}{p+1} - \frac{1}{2\theta} >0$, we obtain an asymptotically decaying variance for any fixed $\alpha>0$.
			\item[Case 3. Severely ill-posed operator] On the other hand, assume that the singular value $s_j^2$ of $B^* B$ decays in an exponential manner, i.e. $s_j^2 \asymp \exp(-2cj)$ for some $c>0$. Then $\mathcal{N}\left(\frac{\alpha}{t}\right) \asymp \frac{1}{2c} \log\left(\frac{t}{\alpha}\right)$ yields an asymptotically decaying variance for any fixed $\alpha>0$.
		\end{description}
	\end{rem}
	
	We summarize both bias and variance bounds and derive the asymptotic behavior of the non-stationary ARM below.
	\begin{thm}\label{thm_1}
		Let Assumptions \ref{assp_1}-\ref{assp_3} hold, then the non-stationary ARM yields MSE estimates
		\begin{enumerate}
			\item If the function $\lambda \mapsto \varphi(\lambda)/\lambda^{p+1}$ is non-increasing, then
			\begin{align*}
			\E\| m(t)-u^{\dag}\|^2 \leq \varphi^2\left(\left(\frac{\alpha}{t}\right)^{\frac{1}{p+1}}\right) + \alpha^{-\frac{1}{p+1}} t^{-\frac{p}{p+1}} \mathcal{N}\left(\frac{\alpha}{t}\right)
			\end{align*}
			for all $0 \leq t \leq T$.
			\item If there is a constant $c<\infty$ with $\varphi(\lambda) \leq c \lambda^{p+1}$ as $\lambda\rightarrow 0$, then
			\begin{align*}
			\E\| m(t)-u^{\dag}\|^2 \leq c \left( \frac{\alpha}{t}\right)^2 + \alpha^{-\frac{1}{p+1}} t^{-\frac{p}{p+1}} \mathcal{N}\left(\frac{\alpha}{t}\right) \qquad\text{for all}\qquad 0 \leq t \leq T.
			\end{align*}
		\end{enumerate}
	\end{thm}
	
	We provide some discussion concerning the above error bounds. In view of Remark \ref{rem_2}, the bound $\alpha^{-\frac{1}{p+1}} t^{-\frac{p}{p+1}} \mathcal{N}\left(\frac{\alpha}{t}\right) $ in Proposition \ref{prop_2} decays faster than the constant bound $\alpha^{-1} \tr \left(\Omega\right)$ if any constant $\alpha$ is fixed. Meanwhile, to obtain a better estimate, we can tune the parameter $\alpha$ with respect to the time variable $t$ which balances both bias and variance. For instance, if the ending time is fixed by $T$, we let
	\begin{align}\label{eq_cor_theta}
	\Theta_{\varphi}\left(\epsilon\right) := \epsilon \varphi \left(\epsilon\right)/ \sqrt{\epsilon \mathcal{N}\left(\epsilon^{p+1}\right)}
	\end{align}
	with $\epsilon := \left(\frac{\alpha}{T}\right)^{\frac{1}{p+1}}$. Then by choosing $\alpha^{\Theta}_*$ where $\epsilon^{\Theta}_*=\left(\frac{\alpha^{\Theta}_*}{T}\right)^{\frac{1}{p+1}}$ is a solution to the equation
	\begin{align}\label{eq_cor_alphastar}
	\Theta_{\varphi}\left(\epsilon^{\Theta}_*\right) = \sqrt{\frac{1}{T}}
	\end{align}
	we obtain the following corollary concerning the a priori parameter choice rule $\alpha^{\Theta}_*$.
	
	\begin{cor}\label{cor_NSARM}
		Suppose $\varphi(\lambda) \prec \lambda^{p+1}$. Let Assumptions \ref{assp_1}-\ref{assp_3} hold and $T$ be the ending time of the non-stationary ARM, if we choose the a priori parameter choice $\alpha^{\Theta}_*$ satisfying (\ref{eq_cor_alphastar}), then
		\begin{align*}
		\E\| m(T)-u^{\dag}\|^2 \leq c \varphi^2\left(\Theta_{\varphi}^{-1}\left(\sqrt{\frac{1}{T}}\right)\right), \quad \textrm{as}\quad T\rightarrow \infty.
		\end{align*}
	\end{cor}
	
	\begin{rem}\label{rem_varphichoice_NSARM}
		In particular, let $\varphi(\lambda) =\lambda^{p+1}$ and assume that the singular value $s_j^2$ of $B^* B$ decays in a polynomial manner, i.e. $s_j^2 \asymp j^{-2\theta}$ for some $\theta>0$ which yields $\mathcal{N}\left(\frac{\alpha}{t}\right) \asymp \left(\frac{\alpha}{t}\right)^{-\frac{1}{2\theta}}$ . Then we obtain
		\begin{align}\label{eq_rem_varphichoice_NSARM}
		\E\| m(T)-u^{\dag}\|^2 & \leq c \left( \frac{\alpha}{T}\right)^2 + \alpha^{-\frac{1}{p+1}} T^{-\frac{p}{p+1}} \mathcal{N}\left(\frac{\alpha}{T}\right) \nonumber \\
		& \leq c T^{-\frac{2}{2+\frac{1}{2\theta}+\frac{1}{p+1}}}
		\end{align}
		by choosing $\alpha^{\Theta}_*= T^{\frac{2+\frac{1}{2\theta}-\frac{p}{p+1}}{2+\frac{1}{2\theta}+\frac{1}{p+1}}}$. As Item 2 of Assumption \ref{assp_2} shows $\frac{p}{p+1} > \frac{1}{2\theta}$, we thus obtain
		\begin{align*}
		T^{-\frac{2}{2+\frac{1}{2\theta}+\frac{1}{p+1}}} \sim o(T^{-2/3}),\quad \textrm{as} \quad T\rightarrow \infty
		\end{align*}
		which is a clear advantage when we use the effective dimension $\mathcal{N}\left(\frac{\alpha}{t}\right)$ in bounding the variance term.
		
		At the same time, the error estimate in (\ref{eq_rem_varphichoice_NSARM}) is the saturation of the non-stationary ARM such that one can not improve the rate by assuming a higher smoothness index function $\varphi$.
	\end{rem}
	
	In the end of current subsection, we compare the non-stationary ARM with the Bayesian approach. As has been proven in \cite{ILLS2017}, the discrete Kalman filter is equivalent to the Bayesian approach where the same optimal error estimate can be obtained under appropriate assumptions. Meanwhile, in current subsection, we also verify that the non-stationary ARM, as a continuous analogue of the Kalman filter, is equivalent to the Bayesian approach as investigated in \cite{LLM} if we let $\delta$ in (\ref{eq_LInP}) obey $\delta =1/\sqrt{T}$ as heuristically discussed in the end of Subsection \ref{se13}.
	
	\subsection{Stationary ARM}
	To obtain the MSE estimate of the stationary ARM \eqref{eq_3dvar}, similar to \eqref{eq_KFmerror}, we derive the error between the posterior mean and the unknown exact solution below
	\begin{align}\label{eq_3dvarerror}
	\left\{
	\begin{array}{l}
	\mathrm{d}(u-\zeta) = - C_0 A^*\Sigma^{-1}A (u-\zeta) dt -  C_0 A^*\Sigma^{-1/2} \mathrm{d}W,  \\
	u(0)-\zeta(0) = u^{\dag}-m_0.
	\end{array}
	\right.
	\end{align}
	We thus calculate the solution of above initial value problem (\ref{eq_3dvarerror}) by
	\begin{align*}
	(u-\zeta)(t) = e^{-C_0 A^* \Sigma^{-1} A t} (u-\zeta)(0) - \int_0^t e^{-\int_s^t C_0 A^* \Sigma^{-1} A \mathrm{d}\tau } C_0 A^* \Sigma^{-1/2} \mathrm{d}W(s).
	\end{align*}
	The bias-variance decomposition then allows us to derive the MSE estimate
	\begin{align*}
	\E\|\zeta(t)-u^{\dag}\|^2 = \|u^{\dag} - \E\zeta(t)\|^2 + \E \|\E\zeta(t)-\zeta(t)\|^2
	\end{align*}
	where
	\begin{align*}
	u^{\dag} - \E\zeta(t)  & : = e^{-C_0 A^* \Sigma^{-1} A t} (u-\zeta)(0), \\
	\E\zeta(t)-\zeta(t) & : = - \int_0^t e^{-\int_s^t C_0 A^* \Sigma^{-1} A \mathrm{d}\tau } C_0 A^* \Sigma^{-1/2} \mathrm{d}W(s).
	\end{align*}
	Similar to the previous subsection, we bound both terms separately.
	
	\begin{prop}\label{prop_3}
		Let Assumptions \ref{assp_1}-\ref{assp_3} hold, then the stationary ARM yields an error bound of the bias term
		\begin{align*}
		\|\E \zeta (t) - u^{\dag}\|^2 \leq c \varphi^2\left(\left(\frac{\alpha}{t}\right)^{\frac{1}{p+1}}\right)\qquad\text{for all}\qquad 0 \leq t \leq T
		\end{align*}
		with the constant $c = \max\{(\nu_0/(p+1))^{\nu_0/(p+1)},1\}$ and $\nu_0$ is the qualification index.
	\end{prop}
	\begin{proof}
		Using Assumptions \ref{assp_1}-\ref{assp_3} we can rewrite the bias term by
		\begin{align*}
		\|u^{\dag} - \E\zeta(t)\| = \|e^{-\frac{t}{\alpha} B^* B} \varphi(K^*K)\| = \sup_{0< \lambda \leq \|B^*B\|} r_{2,\frac{\alpha}{t}}(\lambda)\varphi(\lambda^{1/(p+1)}).
		\end{align*}
		We separate the bias estimation into two cases, namely, $0< \frac{\alpha}{t} \leq \lambda$ and $\frac{\alpha}{t} >\lambda$.
		Notice that the qualification $\nu_0>0$ yields a non-increasing function $\lambda \mapsto \varphi(\lambda^{1/(p+1)})/\lambda^{\nu_0/(p+1)}$. So we derive
		\begin{align*}
		\varphi(\lambda^{1/(p+1)})/\lambda^{\nu_0/(p+1)} \leq \varphi\left(\left(\frac{\alpha}{t}\right)^{1/(p+1)}\right)/\left(\frac{\alpha}{t}\right)^{\nu_0/(p+1)}, \qquad 0 < \frac{\alpha}{t}  \leq \lambda.
		\end{align*}
		Then by the above inequality and (\ref{eq_resi2_error}) we derive
		\begin{align*}
		r_{2,\frac{\alpha}{t}}(\lambda)\varphi(\lambda^{1/(p+1)}) & \leq r_{2,\frac{\alpha}{t}}(\lambda) \lambda^{\nu_0/(p+1)}  \varphi\left(\left(\frac{\alpha}{t}\right)^{1/(p+1)}\right)/\left(\frac{\alpha}{t}\right)^{\nu_0/(p+1)} \\
		& \leq c \left(\frac{\alpha}{t}\right)^{\nu_0/(p+1)} \varphi\left(\left(\frac{\alpha}{t}\right)^{1/(p+1)}\right)/\left(\frac{\alpha}{t}\right)^{\nu_0/(p+1)} = c \varphi\left(\left(\frac{\alpha}{t}\right)^{1/(p+1)}\right)
		\end{align*}
		with the constant $c = \max\{(\nu_0/(p+1))^{\nu_0/(p+1)},1\}$.
		On the other hand, if $\frac{\alpha}{t} > \lambda$ we directly obtain
		\begin{align*}
		r_{2,\frac{\alpha}{t}}(\lambda)\varphi(\lambda^{1/(p+1)}) \leq r_{2,\frac{\alpha}{t}}(\lambda)\varphi\left(\left(\frac{\alpha}{t}\right)^{1/(p+1)}\right) \leq \varphi\left(\left(\frac{\alpha}{t}\right)^{1/(p+1)}\right)
		\end{align*}
		noticing $r_{2,\frac{\alpha}{t}} (\lambda) \leq 1$ if $\frac{\alpha}{t}>0$ and $\lambda\geq 0$.
	\end{proof}
	
	\begin{rem}
		We shall emphasize that the qualifications of non-stationary and stationary ARMs are different as shown in Propositions \ref{prop_1} and \ref{prop_3}. More precisely, non-stationary ARM has a qualification $\nu_0=p+1$ and the stationary one has a qualification $\nu_0=\infty$.
	\end{rem}
	
	The main quantitative difference between the non-stationary and stationary ARMs is provided by the following result.
	\begin{prop}\label{prop_4}
		Let Assumptions \ref{assp_1}-\ref{assp_2} hold, then the stationary ARM yields a bound of the variance
		\begin{align*}
		\E\|\E \zeta(t) - \zeta(t)\|^2 & \leq  \frac{1}{2}\alpha^{-1} \tr \left(\Omega\right) \qquad\text{for all}\qquad 0 \leq t \leq T.
		\end{align*}
	\end{prop}
	\begin{proof}
		Similar to the proof of Proposition \ref{prop_2}, we recall the discussion in Subsection \ref{sec:stochasticinte} and implement Assumption \ref{assp_1} to derive
		\begin{align*}
		\E\|\E \zeta(t) - \zeta(t)\|^2 = \frac{1}{\alpha^2}\int_0^t \tr  \left( e^{-2\frac{t-s}{\alpha}\Omega A^* \Sigma^{-1}A } \Omega^2 A^* \Sigma^{-1}A \right) \mathrm{d}s.
		\end{align*}
		Then Assumption \ref{assp_2} further yields
		\begin{align*}
		\E\|\E \zeta(t) - \zeta(t)\|^2 & = \frac{1}{\alpha^2}\int_0^t \tr  \left( e^{-2\frac{t-s}{\alpha}B^*B } \Omega B^*B \right) \mathrm{d}s.
		\end{align*}
		Let now $\lambda_j>0$ be the eigenvalues of $B^*B$. Then we can compute the trace as a sum, and if we furthermore apply Levy's monotone convergence theorem (exploiting $\lambda_j \geq 0$), we find
		\begin{align*}
		\E\|\E \zeta(t) - \zeta(t)\|^2 & = \frac{1}{\alpha^2}\int_0^t \tr  \left( e^{-2\frac{s}{\alpha}B^*B }  (B^*B)^{\frac{p}{p+1}+1} \right) \mathrm{d}s \\
		& = \frac{1}{\alpha^2} \int_0^t \sum_{j=1}^{\infty} \left[e^{-\frac{2}{\alpha} s \lambda_j} (\lambda_j)^{\frac{p}{p+1}+1}\right] \mathrm{d}s \\
		& = -\frac{1}{2\alpha} \sum_{j=1}^{\infty} (\lambda_j)^{\frac{p}{p+1}} \int_0^t e^{-\frac{2}{\alpha}\lambda_j s } \mathrm{d}\left(-\frac{2}{\alpha}\lambda_j s\right) \\
		& = \frac{1}{2\alpha} \sum_{j=1}^{\infty} (\lambda_j)^{\frac{p}{p+1}} \left(1-e^{-\frac{2}{\alpha} t \lambda_j}\right) \\
		& \leq \frac{1}{2}\alpha^{-1} \tr(\Omega),
		\end{align*}
		where the last line follows from $1-e^{-\frac{2}{\alpha} t \lambda_j} \leq 1$ for all $j$.
	\end{proof}
	
	One may doubt whether it is possible to derive an error bound of the variance with respect to the effective dimension as shown in Proposition \ref{prop_2}. The following calculation confirms that such an upper bound blows up faster than it is in Proposition \ref{prop_4} when the time variable $t$ becomes large.
	Indeed, we could estimate
	\begin{align*}
	\E\|\E \zeta(t) - \zeta(t)\|^2 & = \frac{1}{\alpha^2}\int_0^t \tr  \left( e^{-2\frac{t-s}{\alpha}B^*B } \Omega B^*B \right) \mathrm{d}s \\
	& = \frac{1}{\alpha^2}\int_0^t \tr  \left( e^{-2\frac{t-s}{\alpha}B^*B } \Omega \left(\frac{\alpha}{t}+B^*B\right) \left(\frac{\alpha}{t}+B^*B\right)^{-1} B^*B \right) \mathrm{d}s \\
	& \leq \frac{1}{\alpha^2} \mathcal{N}\left(\frac{\alpha}{t}\right)\int_0^t \sup_{0<\lambda \leq 1} \left|e^{-2\frac{t-s}{\alpha}\lambda} \lambda^{\frac{p}{p+1}} \left(\frac{\alpha}{t}+\lambda\right) \right|  \mathrm{d}s \\
	& \leq  \frac{1}{\alpha^2} \mathcal{N}\left(\frac{\alpha}{t}\right)\int_0^t \left( \frac{\alpha}{t}\sup_{0<\lambda \leq 1} \left|e^{-2\frac{t-s}{\alpha}\lambda} \lambda^{\frac{p}{p+1}} \right| + \sup_{0<\lambda \leq 1} \left|e^{-2\frac{t-s}{\alpha}\lambda} \lambda^{1+\frac{p}{p+1}} \right| \right)  \mathrm{d}s.
	\end{align*}
	Using (\ref{eq_resi2_error}), we derive
	\begin{align*}
	\sup_{0<\lambda \leq 1} \left|e^{-2\frac{t-s}{\alpha}\lambda} \lambda^{\frac{p}{p+1}} \right| & \leq 1/\left(1+2\frac{t-s}{\alpha}\right)^{\frac{p}{p+1}} \\
	& \leq 1/\left(2\frac{t-s}{\alpha}\right)^{\frac{p}{p+1}}.
	\end{align*}
	In particular, we also obtain, with $c=\left(\frac{2p+1}{p+1}\right)^{\frac{2p+1}{p+1}}$
	\begin{align*}
	\sup_{0<\lambda \leq 1} \left|e^{-2\frac{t-s}{\alpha}\lambda} \lambda^{1+\frac{p}{p+1}} \right|
	& \leq c/\left(1+2\frac{t-s}{\alpha}\right)^{\frac{2p+1}{p+1}}
	\end{align*}
	we thus derive
	\begin{align*}
	\E\|\E \zeta(t) - \zeta(t)\|^2 & \leq c_1 \alpha^{-\frac{1}{p+1}} t^{-\frac{p}{p+1}} \mathcal{N}\left(\frac{\alpha}{t}\right) + c_2 \alpha^{-1}  \left(1-\left(1+2\frac{t}{\alpha}\right)^{-\frac{p}{p+1}}\right) \mathcal{N}\left(\frac{\alpha}{t}\right),
	\end{align*}
	with $c_1 = 2^{-\frac{p}{p+1}} (p+1)^{-1}$ and $c_2 = \frac{p}{2(p+1)}\left(\frac{2p+1}{p+1}\right)^{\frac{2p+1}{p+1}}$. Note that $1-\frac1 \gamma \leq \gamma-1$ for all $\gamma \geq 0$. Hence
	\[
	\left(1-\left(1+2\frac{t}{\alpha}\right)^{-\frac{p}{p+1}}\right) \leq \left(1+2\frac{t}{\alpha}\right)^{\frac{p}{p+1}}-1.
	\]
	As the mapping $1+\gamma \mapsto (1+\gamma)^q$ with $q = \frac{p}{p+1} < 1$ is concave, it holds $\left(1+\gamma\right)^q - 1 \leq 1^q + \gamma^q - 1 = \gamma^q$, and thus
	\[
	\left(1-\left(1+2\frac{t}{\alpha}\right)^{-\frac{p}{p+1}}\right) \leq \left(1+2\frac{t}{\alpha}\right)^{\frac{p}{p+1}}-1 \leq \left(2\frac{t}{\alpha}\right)^{\frac{p}{p+1}}.
	\]
	We thus obtain
	\begin{align}\label{eq_stationARM_effectbound}
	\E\|\E \zeta(t) - \zeta(t)\|^2 & \leq c_1 \alpha^{-\frac{1}{p+1}} t^{-\frac{p}{p+1}} \mathcal{N}\left(\frac{\alpha}{t}\right) + c_2 2^{\frac{p}{p+1}} \alpha^{-\frac{2p+1}{p+1}} t^{\frac{p}{p+1}} \mathcal{N}\left(\frac{\alpha}{t}\right)
	\end{align}
	where the second term in the right-hand side blows up faster than a constant function when $t$ increases.
	
	To some extend, it seems intuitive that the stationary ARM does not yield a better bound for the variance, as in \eqref{eq_3dvar} the covariance is fixed, and hence no improvement over time is to be expected. In fact, the bound from Proposition \ref{prop_4} reminds a bit of the classical worst-case bound in deterministic inverse problems, exploiting that the initial covariance $\Omega$ was assumed to be of trace class. However, this is insufficient for minimax optimality in statistical inverse problems and emphasizes the difference between the stationary ARM and Showalter regularization, the latter known to be minimax optimal in many situations (cf. \cite{bhmr07}). 
	
	We summarize both bias and variance bounds and derive the asymptotic behavior of the stationary ARM below.
	\begin{thm}\label{thm_2}
		Let Assumptions \ref{assp_1}-\ref{assp_3} hold, then the stationary ARM yields the MSE estimate
		\begin{align*}
		\E\| \zeta(t)-u^{\dag}\|^2 \leq c \varphi^2\left(\left(\frac{\alpha}{t}\right)^{\frac{1}{p+1}}\right) +  \frac{1}{2}  \alpha^{-1} \tr \left(\Omega\right)\qquad\text{for all}\qquad 0 \leq t \leq T.
		\end{align*}
		with the constant $c = \max\{(\nu_0/(p+1))^{\nu_0/(p+1)},1\}$.
	\end{thm}
	
	Similar to the previous subsection, we provide some discussion concerning the above error bound. To obtain a suitable bound for the MSE in case of the stationary ARM, we can again tune the parameter $\alpha$ to balance both bias and variance. For instance, if the ending time is fixed  sufficiently large by $T$, we let
	\begin{align*}
	\Psi_{\varphi}\left(\epsilon\right) := \epsilon^{\frac{p+1}{2}} \varphi \left(\epsilon\right)
	\end{align*}
	with $\epsilon := \left(\frac{\alpha}{T}\right)^{\frac{1}{p+1}}$. Then by choosing $\alpha^{\Psi}_*$ where $\epsilon^{\Psi}_*=\left(\frac{\alpha^{\Psi}_*}{T}\right)^{\frac{1}{p+1}}$ is a solution to the equation
	\begin{align}\label{eq_cor_alphastar2}
	\Psi_{\varphi}\left(\epsilon^{\Psi}_*\right)  = \sqrt{\frac{1}{T}}
	\end{align}
	we obtain the following corollary concerning the a priori parameter choice rule $\alpha^{\Psi}_*$.
	
	\begin{cor}
		Let Assumptions \ref{assp_1}-\ref{assp_3} hold and $T$ be the sufficiently large ending time of the stationary ARM, if we choose the a priori parameter choice $\alpha^{\Psi}_*$ satisfying (\ref{eq_cor_alphastar2}), then
		\begin{align*}
		\E\| m(T)-u^{\dag}\|^2 \leq c (1+\tr(\Omega))\varphi^2\left(\Psi_{\varphi}^{-1}\left(\sqrt{\frac{1}{T}}\right)\right), \quad \textrm{as}\quad T\rightarrow \infty.
		\end{align*}
	\end{cor}
	
	Comparing the variance terms in Propositions \ref{prop_2} and \ref{prop_4}, we quantitatively observe a difference between non-stationary and stationary ARMs, where the use of the effective dimension leads to a better variance bound when implementing the non-stationary ARM. Some discussion concerning the corresponding MSE bounds is provided below.
	
	Noticing that both $\Theta_{\varphi}(\epsilon)$ and $\Psi_{\varphi}(\epsilon)$ are non-decreasing continuous functions and $\lim_{\epsilon\searrow 0}\Theta_{\varphi}(\epsilon) = 0$, $\lim_{\epsilon\searrow 0}\Psi_{\varphi}(\epsilon) = 0$, we obtain $\lim_{T\nearrow \infty}\Theta^{-1}_{\varphi}\left(\sqrt{\frac{1}{T}}\right) = 0$ and $\lim_{T\nearrow \infty}\Psi^{-1}_{\varphi}\left(\sqrt{\frac{1}{T}}\right) = 0$, namely $\lim_{T\nearrow \infty}\epsilon^{\Theta}_* = \lim_{T\nearrow \infty}\epsilon^{\Psi}_* = 0$. Meanwhile, for all $\epsilon \in (0,\infty)$, there holds
	\begin{align*}
	\frac{\Psi_{\varphi}(\epsilon)}{\Theta_{\varphi}(\epsilon)} = \epsilon^{\frac{p}{2}} \sqrt{\mathcal{N}(\epsilon^{p+1})} \leq \sqrt{\tr(\Omega)},
	\end{align*}
	where the latter estimate follows from \eqref{eq_resi_error} and the computation
	\begin{align*}
	\mathcal{N}(\epsilon^{p+1}) &  = \tr\left( \left(\epsilon^{p+1} + B^* B\right)^{-1} B^*B \right) \\
	& \leq \frac{1}{\epsilon^{p+1}} \left\|r_{1,\epsilon^{p+1}}(B^*B) (B^*B)^{\frac{1}{p+1}}\right\| \tr\left( (B^*B)^{\frac{p}{p+1}}\right) \\
	& \leq \frac{1}{\epsilon^p}  \tr\left(\Omega\right).
	\end{align*}
	Hence, whenever
	\begin{equation}\label{eq:cond}
	\epsilon^{\frac{p}{2}} \sqrt{\mathcal{N}(\epsilon^{p+1})} \to 0\qquad\text{as}\qquad \epsilon \to 0,
	\end{equation}
	we have $\lim_{\epsilon\searrow 0} \Psi_{\varphi}(\epsilon)/\Theta_{\varphi}(\epsilon) = 0$ and consequently
	\begin{align*}
	\lim_{T\nearrow \infty} \frac{\Theta^{-1}_{\varphi}\left(\sqrt{\frac{1}{T}}\right)}{\Psi^{-1}_{\varphi}\left(\sqrt{\frac{1}{T}}\right)} = 0.
	\end{align*}
	This shows that for H\"{o}lder-type source conditions with $\nu \leq \nu_0 = p+1$ (where according to Remark \ref{rem_2} \eqref{eq:cond} is satisfied) the obtained rates for non-stationary ARM in Theorem \ref{thm_1} are better than the ones from Theorem \ref{thm_2} for the stationary ARM. However, under higher order source conditions, the stationary ARM will yield a better convergence rate due to its infinite qualification, as discussed in the following Remark.
	
	\begin{rem}\label{rem_varphichocie_SARM}
		We take two special choices of $\varphi(\lambda)$. The first one considers $\varphi(\lambda) =\lambda^{p+1}$ which yields
		\begin{align*}
		\E\| m(T)-u^{\dag}\|^2 \leq c (1+\tr(\Omega)) T^{-\frac{2}{3}}, \quad \textrm{as}\quad T\rightarrow \infty.
		\end{align*}
		Meanwhile, if $\varphi(\lambda) =\lambda^{2(p+1)}$ we obtain
		\begin{align*}
		\E\| m(T)-u^{\dag}\|^2 \leq c (1+\tr(\Omega)) T^{-\frac{4}{5}}, \quad \textrm{as}\quad T\rightarrow \infty.
		\end{align*}
		Because of high qualification of the stationary ARM, we can obtain a better error estimate if the unknown solution is sufficiently smooth. In view of the error estimate in Remark \ref{rem_varphichoice_NSARM} for the non-stationary ARM, a better error estimate is only available if either the modified forward operator or the prior is smooth enough.
	\end{rem}
	
	\newlength{\fwidth}
	\newlength{\fheight}
	
	\section{Numerics}\label{se4}
	
	In this section we will describe a possible implementation of non-stationary and stationary ARMs \eqref{eq_KFm} and \eqref{eq_3dvar}. Therefore note that both give implicit formulas for the mean functions $m \left(t\right)$ and $\zeta\left(t\right)$ respectively as solutions of initial value problems. In principle, these initial value problems are of stochastic nature (due to the random noise contributions), but we will however solve them by the standard explicit Euler method for the solution of deterministic ODE's, as the randomness is purely due to the data and hence predetermined. Let therefore the process $\left(z\left(t\right)\right)_{t \in \left[0,T\right]}$ as in \eqref{eq_consys} be given and let $h > 0$ be a time step. As mentioned before, we have $\delta = 1/\sqrt{T}$ referring to the noise level in the original problem \eqref{eq_LInP}.
	
	\paragraph{Non-stationary ARM}
	
	The explicit Euler method for the initial value problem \eqref{eq_KFm} leads to the iteration
	\begin{align*}
	m_n &:= m_{n-1} + C_{n-1}A^* \Sigma^{-1} \left(z\left(nh\right) - z\left((n-1)h\right)  -hAm_{n-1}\right), \\
	C_n &:= C_{n-1} - h C_{n-1} A^* \Sigma^{-1} A C_{n-1}
	\end{align*}
	with $n \in \left\{1,..., \left\lfloor \frac{T}{h}\right\rfloor\right\}$.
	
	Note, that in view of \eqref{eq_resi_error}, this method should (for suitably chosen $h>0$) lead to results comparable with the Tikhonov regularization applied to the final datum $z \left(T\right)$, i.e.
	\[
	\hat m_\epsilon := \left(\epsilon + K^*K\right)^{-1} K^* z\left(T\right)
	\]
	with $\epsilon = \alpha/T$ and the pre-whitened operator $K = \Sigma^{-1/2} A$.
	
	\paragraph{Stationary ARM}
	
	The explicit Euler method for the initial value problem \eqref{eq_3dvar} leads to the iteration
	\begin{align*}
	\zeta_n &:= \zeta_{n-1} + C_{0}A^* \Sigma^{-1} \left(z\left(nh\right) - z\left((n-1)h\right)  -hA\zeta_{n-1}\right)
	\end{align*}
	with $n \in \left\{1,..., \left\lfloor \frac{T}{h}\right\rfloor\right\}$.
	
	Note, that in view of \eqref{eq_resi2_error}, this method should (for suitably chosen $h>0$) lead to results comparable with the Showalter regularization applied to the final datum $z \left(T\right)$, i.e.
	\[
	\hat \zeta_\epsilon := \left(1- \exp\left(\frac{1}{\epsilon}K^*K\right)\right) \left(K^*K\right)^{-1} K^*z\left(T\right)
	\]
	again with $\epsilon = \alpha/T$ and the pre-whitened operator $K = \Sigma^{-1/2} A$.

	
	In current section, we consider a prototypical inverse problem as an illustration. In particular, both non-stationary and stationary ARMs are implemented to solve the following second anti-derivative problem. Let $A$ be given as the integral operator $A : \mathbf L^2 \left(\left[0,1\right]\right) \to \mathbf L^2 \left(\left[0,1\right]\right)$ of the first kind
	\begin{equation}\label{eq:op}
	\left(Au\right) \left(x\right):= \int_0^1 k\left(x,y\right)u\left(y\right) \,\mathrm d y, \qquad x \in \left[0,1\right]
	\end{equation}
	with the kernel $k\left(x,y\right) = \min \left\{ x \left(1-y\right), y \left(1-x\right) \right\}$, $x,y \in \left[0,1\right]$, see \cite{hw14,w18}. We discretize the operator $A$ using the composite midpoint rule. Concerning the solution $u^\dagger$, we consider three different choices of increasing smoothness. 	
	To obtain numerical results, we consider $T \in \left\{100 \cdot 2^i~\big|~0\leq i \leq 9\right\}$, and compute empirical values of the root mean-squared error (RMSE) $\sqrt{\E\left\Vert \hat u - u^\dagger\right\Vert^2 }$ by $M = 100$ Monte-Carlo repetitions. The parameter $\alpha$ is then chosen as the value in $\left\{0.1 \cdot 2^j~\big|~0\leq j \leq 39\right\}$ which minimizes the RMSE for the corresponding value of $T$. The remaining parameters are set as the discretization level $n = 512$, the time step $h = T/100$ and the noise covariance $\Sigma = I$. As for the initial covariance, we set $C_0 = (A^*A)^{p/p+1}$ noticing $\Sigma=I$, where $p$ can be chosen as zero in the finite dimensional numerical examples and consequently $C_0 = I$.
	
	\begin{rem}
		In this setting, the minimax rate of convergence for the RMSE for $u^\dagger\in H^s \left[0,1\right]$ with the standard Sobolev space
		\[
		H^s \left[0,1\right] := \left\{u \in L^2 \left[0,1\right]~\big|~ \sum_{k=1}^\infty \left(1+k^2\right)^s u_k^2 < \infty \right\}
		\]
		where $u_k$ denote the Fourier coefficients of $u$ is given by $\mathcal O \left(\delta^{\frac{s}{s + \frac52}}\right)$, see e.g. \cite{dm18,c08}. This makes use of the fact that the singular values $\sigma_k$ of our operator $A$ in \eqref{eq:op} decay like $\sigma_k \sim k^{-2}$. Using this, we can furthermore see that
		\[
		u^\dagger \in H^s\left[0,1\right]\qquad\text{if and only if}\qquad u^\dagger \in R \left(\left(A^*A\right)^{\frac{s}{4}}\right).
		\]
		This shows that the minimax rate of convergence for the RMSE under the source condition $u^\dagger \in R \left(\left(A^*A\right)^\nu\right)$ is $\mathcal O \left(\delta^{\frac{4\nu}{4\nu + \frac{5}{2}}}\right)$.
	\end{rem}
	
	\begin{example}[Exact solutions with low smoothness]\label{example_1}
		The first example chooses
		$u^\dagger$ as the hat function
		\[
		u^\dagger\left(x\right) = \begin{cases} x & \text{if }0 \leq x \leq \frac12, \\ 1-x & \text{if }\frac12 \leq x \leq 1, \end{cases}
		\]
		which leads to the exact datum
		\[
		\left(Au^\dagger \right)(x) = \begin{cases} -\frac{x\left(4x^2-3\right)}{24} & \text{if }0 \leq x \leq \frac12,\\ \frac{\left(x-1\right)\left(4x^2-8x+1\right)}{24} & \text{if } \frac12 \leq x \leq 1.\end{cases}
		\]
		It can readily be seen, that this function satisfies a source condition of the form $u^\dagger \in R \left(\left(A^*A\right)^{\frac38 - \varepsilon}\right)$ for any $\varepsilon > 0$, and consequently the optimal rate of convergence of the RMSE for this choice of $A$ and $u^\dagger$ in \eqref{eq_LInP} is $\mathcal O \left(\delta^{\frac38-\varepsilon}\right)$, i.e. $\mathcal O \left(\left(\frac{1}{\sqrt{T}}\right)^{\frac38-\varepsilon}\right)$ for any $\varepsilon>0$ and $p=0$ (see e.g. \cite{lw20} for details).
	\end{example}
	
	In Figure \ref{fig:antiderivative1}, we present the empirical RMSEs and optimal values of the regularization parameter $\alpha$ driven by 100 runs. As can be observed in the left panel, the empirical RMSE by the non-stationary ARM fits the theoretical rate accurately. Meanwhile, the empirical RMSE by the stationary ARM decays asymptotically when $T$ increases but with a smaller slope compared with the non-stationary ARM, which numerically verifies the quantitative difference between both methods. We shall mention that in Example \ref{example_1} we have chosen the prior with $p=0$ since the discretization of the forward operator yields a finite dimensional matrix. We will tune this value in the high smoothness example to visualize the improvement induced by the smooth initial covariance.
	
	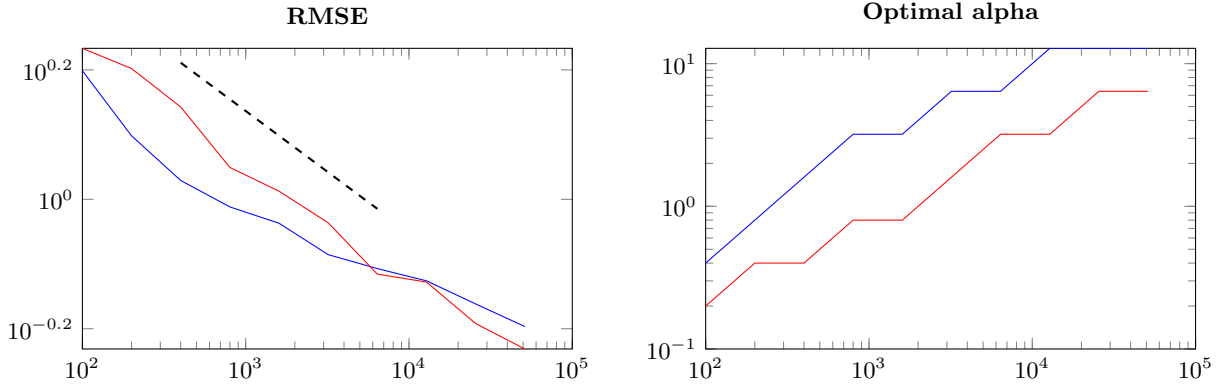
\begin{figure}[!htb]
		\footnotesize
		\setlength{\fwidth}{.39\textwidth}
		\setlength{\fheight}{4cm}
		\centering
		\begin{tabular}{rr}
			\begin{tikzpicture}[baseline]

\begin{axis}[%
width=\fwidth,
height=\fheight,
scale only axis,
xmode=log,
xmin=100,
xmax=100000,
xminorticks=true,
ymode=log,
ymin=0.587152623969964,
ymax=1.71137744932664,
yminorticks=true,
title style={font=\bfseries},
title={RMSE}
]
\addplot [color=blue]
  table[row sep=crcr]{%
100	1.58010237583868\\
200	1.25344257206863\\
400	1.0694029357059\\
800	0.973318075810553\\
1600	0.918387446532354\\
3200	0.821359098386648\\
6400	0.781629571239895\\
12800	0.748322257767139\\
25600	0.689230368461008\\
51200	0.635930716937176\\
};
\label{stat}

\addplot [color=red]
  table[row sep=crcr]{%
100	1.71137744932664\\
200	1.59246901027884\\
400	1.3889706323038\\
800	1.12009575395568\\
1600	1.02943129369949\\
3200	0.919844543431405\\
6400	0.766436388681523\\
12800	0.744974839987077\\
25600	0.642740187706653\\
51200	0.587152623969964\\
};
\label{non_stat}

\addplot [color=black,dashed,thick]
  table[row sep=crcr]{%
400	1.62586228156059\\
800	1.42771207223012\\
1600	1.2537112056226\\
3200	1.10091650667945\\
6400	0.966743496623206\\
};
\label{opt}

\end{axis}
\end{tikzpicture} & \begin{tikzpicture}[baseline]

\begin{axis}[%
width=\fwidth,
height=\fheight,
scale only axis,
xmode=log,
xmin=100,
xmax=100000,
xminorticks=true,
ymode=log,
ymin=0.1,
ymax=12.8,
yminorticks=true,
title style={font=\bfseries},
title={Optimal alpha}
]
\addplot [color=blue]
  table[row sep=crcr]{%
100	0.4\\
200	0.8\\
400	1.6\\
800	3.2\\
1600	3.2\\
3200	6.4\\
6400	6.4\\
12800	12.8\\
25600	12.8\\
51200	12.8\\
};

\addplot [color=red]
  table[row sep=crcr]{%
100	0.2\\
200	0.4\\
400	0.4\\
800	0.8\\
1600	0.8\\
3200	1.6\\
6400	3.2\\
12800	3.2\\
25600	6.4\\
51200	6.4\\
};

\end{axis}
\end{tikzpicture}\\
		\end{tabular}
		\caption{(Low smoothness) Empirical RMSEs (left) and optimal values of the regularization parameter $\alpha$ (right) for solving the second anti-derivative problems (simulated from $100$ runs) with $u^\dagger$ as in Example \ref{example_1}. Shown are stationary ARM (\ref{stat}), non-stationary ARM with $p = 0$ (\ref{non_stat}), cf. Assumption \ref{assp_2}, as well as the optimal rate of convergence (\ref{opt}).}
		\label{fig:antiderivative1}
	\end{figure}

	\begin{example}[Exact solutions with intermediate smoothness]\label{example_2}
		$u^\dagger$ is chosen as a scaled version of $A$ applied to the above hat function, i.e.
		\[
		u^\dagger\left(x\right) = 10\begin{cases} -\frac{x\left(4x^2-3\right)}{24} & \text{if }0 \leq x \leq \frac12,\\ \frac{\left(x-1\right)\left(4x^2-8x+1\right)}{24} & \text{if } \frac12 \leq x \leq 1.\end{cases}
		\]
		This leads to the exact datum
		\[
		\left(Au^\dagger \right)(x) = 10\begin{cases} \frac{x\left(16x^4 - 40x^2 + 25\right)}{1920} & \text{if }0 \leq x \leq \frac12,\\ \frac{-16x^5 + 80x^4-120x^3 + 40x^2 + 15x + 1}{1920} & \text{if } \frac12 \leq x \leq 1.\end{cases}
		\]
		As $A$ is self-adjoint, we obtain $u^\dagger \in R \left(\left(A^*A\right)^{\frac78 - \varepsilon}\right)$ for any $\varepsilon > 0$, and consequently the optimal rate of convergence of the RMSE for this choice of $A$ and $u^\dagger$ in \eqref{eq_LInP} is $\mathcal O \left(\delta^{\frac{7}{12}-\varepsilon}\right)$, i.e. $\mathcal O \left(\left(\frac{1}{\sqrt{T}}\right)^{\frac{7}{12}-\varepsilon}\right)$ for any $\varepsilon>0$ and $p=0$.
	\end{example}
	
	In Figure \ref{fig:antiderivative2}, we present the empirical RMSEs and optimal values of the regularization parameter $\alpha$ driven by 100 runs for Example \ref{example_2}. As can be observed in the left panel, the empirical RMSE by both non-stationary and stationary ARMs fits the theoretical rate accurately. At the same time, the increased smoothness allows better performance of the stationary ARM.
	
	\begin{figure}[!htb]
		\footnotesize
		\setlength{\fwidth}{.39\textwidth}
		\setlength{\fheight}{4cm}
		\centering
		\begin{tabular}{rr}
			\begin{tikzpicture}[baseline]

\begin{axis}[%
width=\fwidth,
height=\fheight,
scale only axis,
xmode=log,
xmin=100,
xmax=100000,
xminorticks=true,
ymode=log,
ymin=0.115133441400089,
ymax=1.74207093857127,
yminorticks=true,
title style={font=\bfseries},
title={RMSE}
]
\addplot [color=blue]
  table[row sep=crcr]{%
100	1.28021601213794\\
200	0.988229486047453\\
400	0.731971655643609\\
800	0.604877305562461\\
1600	0.469436637410319\\
3200	0.342558397609746\\
6400	0.25695117221629\\
12800	0.185088888093272\\
25600	0.148072314318727\\
51200	0.115133441400089\\
};

\addplot [color=red]
  table[row sep=crcr]{%
100	1.72852284638097\\
200	1.34956533820321\\
400	1.14351534470454\\
800	0.896520319746604\\
1600	0.780324704903763\\
3200	0.645841395426589\\
6400	0.483977905966007\\
12800	0.420474742115561\\
25600	0.359923531618329\\
51200	0.283024494305946\\
};

\addplot [color=black,dashed,thick]
  table[row sep=crcr]{%
400	1.74207093857127\\
800	1.42319831358691\\
1600	1.16269285879816\\
3200	0.949870914681696\\
6400	0.776004383041341\\
};

\end{axis}
\end{tikzpicture}%
 & \begin{tikzpicture}[baseline]

\begin{axis}[%
width=\fwidth,
height=\fheight,
scale only axis,
xmode=log,
xmin=100,
xmax=100000,
xminorticks=true,
ymode=log,
ymin=0.1,
ymax=102.4,
yminorticks=true,
title style={font=\bfseries},
title={Optimal alpha}
]
\addplot [color=blue]
  table[row sep=crcr]{%
100	0.4\\
200	0.8\\
400	1.6\\
800	3.2\\
1600	3.2\\
3200	6.4\\
6400	12.8\\
12800	25.6\\
25600	51.2\\
51200	102.4\\
};

\addplot [color=red]
  table[row sep=crcr]{%
100	0.2\\
200	0.4\\
400	0.4\\
800	0.8\\
1600	1.6\\
3200	1.6\\
6400	3.2\\
12800	6.4\\
25600	6.4\\
51200	12.8\\
};

\end{axis}
\end{tikzpicture}\\
		\end{tabular}
		\caption{(Intermediate smoothness) Empirical RMSEs (left) and optimal values of the regularization parameter $\alpha$ (right) for solving the second anti-derivative problems (simulated from $100$ runs) with $u^\dagger$ as in Example \ref{example_2}. Shown are stationary ARM (\ref{stat}), non-stationary ARM with $p = 0$ (\ref{non_stat}), cf. Assumption \ref{assp_2}, as well as the optimal rate of convergence (\ref{opt}).}
		\label{fig:antiderivative2}
	\end{figure}
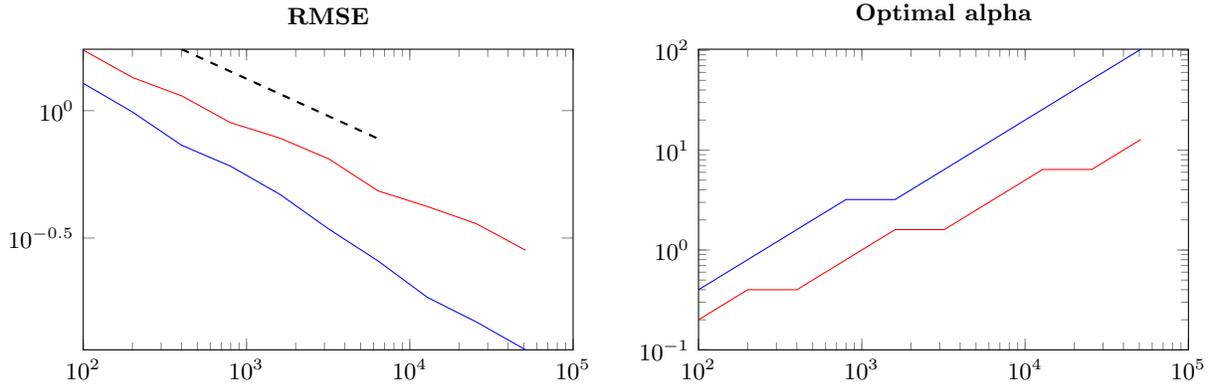

	\begin{example}[Exact solutions with high smoothness]\label{example_3}
		$u^\dagger$ is chosen as a scaled version of $A^*A$ applied to the above hat function, i.e.
		\[
		u^\dagger(x) = 100\begin{cases} \frac{x\left(16x^4 - 40x^2 + 25\right)}{1920} & \text{if }0 \leq x \leq \frac12,\\ \frac{-16x^5 + 80x^4-120x^3 + 40x^2 + 15x + 1}{1920} & \text{if } \frac12 \leq x \leq 1.\end{cases}
		\]
		This leads to the exact datum
		\[
		\left(Au^\dagger \right)(x) =  100\begin{cases} \frac{427 x-700 x^3+336 x^5-64 x^7}{322560}& \text{if }0 \leq x \leq \frac12,\\
		\frac{-1+441 x-84 x^2-420 x^3-560 x^4+1008 x^5-448 x^6+64 x^7}{322560} & \text{if } \frac12 \leq x \leq 1.\end{cases}
		\]
		We obtain $u^\dagger \in R \left(\left(A^*A\right)^{\frac{11}{8} - \varepsilon}\right)$ for all $\varepsilon > 0$, and consequently the optimal rate of convergence of the RMSE for this choice of $A$ and $u^\dagger$ in \eqref{eq_LInP} is $\mathcal O \left(\delta^{\frac{11}{16}-\varepsilon}\right)$, i.e. $\mathcal O \left(\left(\frac{1}{\sqrt{T}}\right)^{\frac{11}{16}-\varepsilon}\right)$ for any $\varepsilon>0$. Note that, when $p=0$, this function is smoother than the qualification of the non-stationary ARM allows for, as the best possible rate to be obtained by Tikhonov regularization in the model \eqref{eq_LInP} with the operator $A$ described here is $\mathcal O \left(\delta^{\frac{8}{13}}\right)$, i.e. $\mathcal O \left(\left(\frac{1}{\sqrt{T}}\right)^{\frac{8}{13}}\right)$.
		
	\end{example}

	In Figure \ref{fig:antiderivative3}, we present the empirical RMSEs and optimal values of the regularization parameter $\alpha$ driven by 100 runs for Example \ref{example_3}. It can be seen from the left panel, that the non-stationary ARM with $p = 0$ saturates at the rate $\mathcal O \left(\left(\frac{1}{\sqrt{T}}\right)^{\frac{8}{13}}\right)$ and does not yield order-optimal results. Hence, we also implemented the non-stationary ARM with $p = \frac12$, which yields order optimal convergence as visible in the left panel of Figure \ref{fig:antiderivative3}. Furthermore, the stationary ARM fits the theoretical rate accurately.
	
	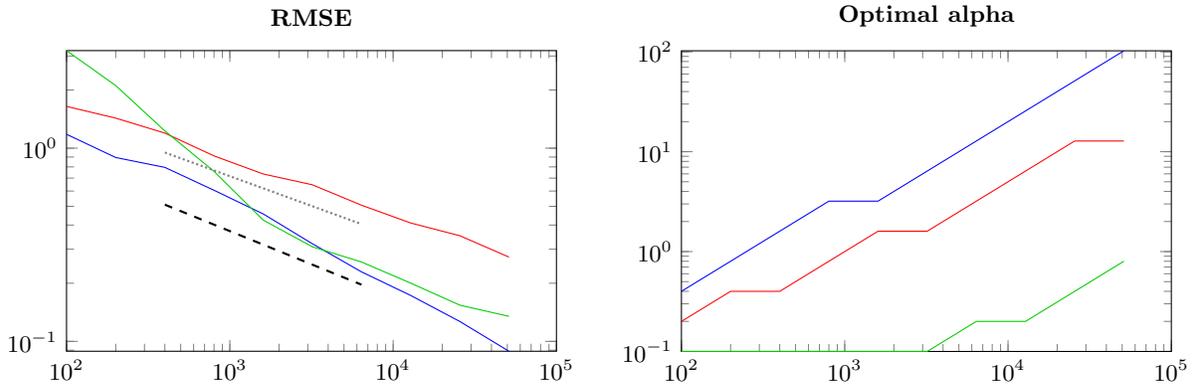
\begin{figure}[!htb]
		\footnotesize
		\setlength{\fwidth}{.39\textwidth}
		\setlength{\fheight}{4cm}
		
		\centering
		
		\begin{tabular}{rr}
			
			\begin{tikzpicture}[baseline]

\begin{axis}[%
width=\fwidth,
height=\fheight,
scale only axis,
xmode=log,
xmin=100,
xmax=100000,
xminorticks=true,
ymode=log,
ymin=0.0888990669829668,
ymax=3.19914768391837,
yminorticks=true,
title style={font=\bfseries},
title={RMSE}
]
\addplot [color=blue]
  table[row sep=crcr]{%
100	1.18009910928714\\
200	0.89486885871978\\
400	0.795480645859441\\
800	0.605541573177289\\
1600	0.456933519975692\\
3200	0.320203943519684\\
6400	0.229024012543153\\
12800	0.172903074559242\\
25600	0.126792662300592\\
51200	0.0888990669829668\\
};

\addplot [color=red]
  table[row sep=crcr]{%
100	1.64522024254779\\
200	1.43076515706874\\
400	1.19677476559771\\
800	0.911909072878402\\
1600	0.73539191440063\\
3200	0.646544808985822\\
6400	0.506777567545434\\
12800	0.410031027988726\\
25600	0.352286036681845\\
51200	0.27305149735181\\
};

\addplot [color=green!80!black]
  table[row sep=crcr]{%
100	3.19914768391837\\
200	2.10937207947268\\
400	1.22859327935497\\
800	0.758684920792049\\
1600	0.424633208451313\\
3200	0.308690528700685\\
6400	0.257627114702508\\
12800	0.200221709818934\\
25600	0.154010657307047\\
51200	0.134742606160389\\
};
\label{non_stat_improved}

\addplot [color=black,dashed,thick]
  table[row sep=crcr]{%
400	0.510037219278843\\
800	0.401904443937774\\
1600	0.316696852604834\\
3200	0.249554086705587\\
6400	0.196646230233199\\
};

\addplot [color=gray, densely dotted,thick]
  table[row sep=crcr]{%
400	0.94954905718623\\
800	0.767172086801012\\
1600	0.619823700852972\\
3200	0.500776066737583\\
6400	0.40459354599067\\
};
\label{sat}

\end{axis}
\end{tikzpicture}%
 & \begin{tikzpicture}[baseline]

\begin{axis}[%
width=\fwidth,
height=\fheight,
scale only axis,
xmode=log,
xmin=100,
xmax=100000,
xminorticks=true,
ymode=log,
ymin=0.1,
ymax=102.4,
yminorticks=true,
title style={font=\bfseries},
title={Optimal alpha},
]
\addplot [color=blue]
  table[row sep=crcr]{%
100	0.4\\
200	0.8\\
400	1.6\\
800	3.2\\
1600	3.2\\
3200	6.4\\
6400	12.8\\
12800	25.6\\
25600	51.2\\
51200	102.4\\
};

\addplot [color=red]
  table[row sep=crcr]{%
100	0.2\\
200	0.4\\
400	0.4\\
800	0.8\\
1600	1.6\\
3200	1.6\\
6400	3.2\\
12800	6.4\\
25600	12.8\\
51200	12.8\\
};

\addplot [color=green!80!black]
  table[row sep=crcr]{%
100	0.1\\
200	0.1\\
400	0.1\\
800	0.1\\
1600	0.1\\
3200	0.1\\
6400	0.2\\
12800	0.2\\
25600	0.4\\
51200	0.8\\
};

\end{axis}
\end{tikzpicture}
			
		\end{tabular}
		
		\caption{(High smoothness) Empirical RMSEs (left) and optimal values of the regularization parameter $\alpha$ (right) for solving the second anti-derivative problems (simulated from $100$ runs) with $u^\dagger$ as in Example \ref{example_3}. Shown are stationary ARM (\ref{stat}), non-stationary ARM with $p= 0$ (\ref{non_stat}), cf. Assumption \ref{assp_2}, and non-stationary ARM with $p = \frac12$ (\ref{non_stat_improved}), as well as the optimal rate of convergence in the corresponding example with $p=1/2$ (\ref{opt}) and the saturation rate of non-stationary ARM with $p = 0$ (\ref{sat}).}
		\label{fig:antiderivative3}
		
	\end{figure}
	
	\section{Conclusion and future extensions}\label{se5}
	In this paper we investigate asymptotical regularization for linear inverse problems in presence of white noise. By arguing that the available data often arises from subsequent identical measurements, we analyze two different methods for a continuous artificial dynamical system related to the original problem. Those are the non-stationary ARM, which is a consequence of the Kalman-Bucy filter where the posterior covariance varies with respect to the time variable, and the stationary ARM, which is a consequence of the 3DVAR with a fixed posterior covariance. This bridges a gap between regularization theory and data assimilation. Both methods have the advantage that they can be applied in an online fashion to real world problems, whereas standard methods from regularization theory can only be applied after the final datum has been measured.
	
	Methodologically, we derive error bounds for both methods by carefully treating the variance part which is given in terms of an infinite-dimensional stochastic integral against a standard Wiener process. Our theoretical results reveal that the non-stationary ARM is comparable to Tikhonov regularization, and that the derived convergence rates are minimax optimal. Meanwhile, the stationary ARM is comparable to the Showalter regularization, and our error bound seems sub-optimal. From our viewpoint it is not clear if this sub-optimality results from our analysis or the subsequent formulation of the underlying model \eqref{eq_consys}. Nevertheless, the high qualification of the stationary ARM is able to provide better error bounds if the unknown exact solution is sufficiently smooth. Numerical examples confirm these theoretical predictions.
	
	As of now, we have only considered a priori parameter choice rules for the tuning parameter $\alpha$ in the initial covariance. The a posteriori choice of $\alpha$ remains an interesting topic for future research, as those come with two difficulties: On the one hand, in our current formulation $\alpha$ has to be chosen in a preparation step before data comes in. Hence, $\alpha$ cannot be chosen depending on the data. However, one could think of an adaptive formulation, which allows to change $\alpha$ over time depending on the data. On the other hand, standard approaches such as the discrepancy principle or Lepski{\u\i}'s balancing principle might be applicable, but require a completely different analysis (e.g. with a.s. bounds instead of MSE bounds as derived in this paper).
	
	It might also be interesting to consider higher-order asymptotical regularization methods as recently treated in the deterministic setting in \cite{BDES2018,ZH2018}. There, the first-order governing ordinary differential equation is replaced by a high order one to reduce computational costs. The realization of such high-order extensions in present of the Wiener process can enrich the development of approaches in data assimilation.
	
	Finally, our current numerical examples focus on moderately ill-posed problems where the quantitative difference between the non-stationary and stationary ARMs can be visualized. Meanwhile, more numerical evidences could be carried out to support the theoretical predictions including severely ill-posed problems, where it is to be expected that the bias term dominates the MSE.
	
	\section*{Acknowledgments}
	SL is supported by NSFC (No.11925104), Program of Shanghai Academic/Technology Research Leader (19XD1420500) and National Key Research and Development Program of China (No. 2017YFC1404103). FW gratefully acknowledges financial support by the German Research Foundation DFG through subproject A07 of CRC 755. This project has been initiated during a stay of FW in Shanghai, which was partially financed by the CRC 755. We are furthermore grateful to Peter Math\'e and Housen Li for careful proof-reading of the paper and several helpful comments.


\begin{thebibliography}{99}
		
		\bibitem{BCLZ2014}
		Bao G.; Chow S. N.; Li P. J. and Zhou H. M. {\it An inverse random source problem for the Helmholtz equation.} Math. Comp. 83 (2014), no. 285, 215--233.
		
		
		\bibitem{BHM2007}
		Bauer F.; Hohage T.; Munk A. {\it Iteratively regularized Gauss-Newton method for nonlinear inverse problems with random noise}. SIAM J. Numer. Anal. 47 (2009), no. 3, 1827--1846.
		
		\bibitem{bhmr07}
		Bissantz N.; Hohage T.; Munk A.; Ruymgaart F. {\it Convergence rates of general regularization methods for statistical inverse problems and applications}. SIAM J. Numer. Anal. 45 (2007), no. 6, 2610--2636.
		
		\bibitem{BDES2018}		
		Bot R.; Dong G.; Elbau P. and Scherzer O. {\it Convergence rates of first and higher order dynamics for solving linear ill-posed problems}, arXiv:1812.09343
		
		\bibitem{B1997}
		Bhatia R. {\it Matrix analysis}. Graduate Texts in Mathematics, 169. Springer--Verlag, New York, 1997. xii+347 pp.
		
		\bibitem{bhr18}
		Blanchard G.; Hoffmann M. and Reiss M. {\it Optimal adaptation for early stopping in statistical inverse problems}. SIAM/ASA J. Uncertain. Quantif. 6 (2018), no. 3, 1043--1075.
		
		\bibitem{bsww19}
		Bl\"oker D.; Schillings C.; Wacker P. and Weissmann S. {\it Well posedness and convergence analysis	of the ensemble Kalman inversion}. Inverse Problems 35 (2019) 085007 (32pp).
		
		\bibitem{BHMR2007}
		Bissantz N.; Hohage T.; Munk A. and Ruymgaart F. {\it Convergence rates of general regularization methods for statistical inverse problems and applications}. SIAM J. Numer. Anal. 45 (2007), no. 6, 2610--2636.
		
		\bibitem{c08}
		Cavalier, L. {\it Nonparametric statistical inverse problems.} Inverse Problems 24 (2008), no. 3, 034004, 19 pp.
		
		\bibitem{chkp19}
		Clason C.; Helin T.; Kretschmann R. and Piiroinen P. {\it Generalized modes in Bayesian inverse problems}. SIAM/ASA J. Uncertain. Quantif. 7 (2019), no. 2, 652--684.
		
		\bibitem{DLC2018}
		Ding L.; Lu S. and Cheng J. {\it Weak-norm posterior contraction rate of the 4DVAR method for linear severely ill-posed problems}. J. Complexity 46 (2018), 1--18.
		
		\bibitem{dm18}
		Ding  L.; Math\'e  P. {\it Minimax rates for statistical inverse problems under general source conditions}. Comput. Methods Appl. Math. 18 (2018), no. 4, 603--608.
		
		\bibitem{DH2014}
		Dunker F. and Hohage T. {\it On parameter identification in stochastic differential equations by penalized maximum likelihood.} Inverse Problems 30 (2014), no. 9, 095001, 20 pp.
		
		
		\bibitem{EHN1996}
		Engl H.~W.; Hanke M. and Neubauer A. {\it Regularization of inverse problems}. Mathematics and its Applications, 375. Kluwer Academic Publishers Group, Dordrecht, 1996. viii+321 pp.
		
		\bibitem{FS2012}
		Florens J. and Simoni A. {\it Regularizing priors for linear inverse problems}. Scand. J. Stat.,39 (2012), 214--235.
		
		\bibitem{gm11}
		Gawarecki L. and Mandrekar V. {\it Stochastic Differential Equations in Infinite Dimensions with Applications to Stochastic
			Partial Differential Equations}. Springer-Verlag Berlin Heidelberg 2011.
		
		\bibitem{hjp20}
		Harrach B.; Jahn T. and Potthast R. {\it Beyond the Bakushinskii veto: Regularising linear inverse problems without knowing the noise distribution}.
		arXiv preprint: 1811.06721.
		
		\bibitem{Hohage1997}
		Hohage T. {\it Logarithmic convergence rates of the iteratively regularized Gauss-Newton method for an inverse potential and an inverse scattering problem.} Inverse Problems 13 (1997), no. 5, 1279--1299.
		
		\bibitem{hw14}
		Hohage T. and Werner F. {\it Convergence rates for Inverse Problems with Impulsive Noise. }
		SIAM J. Numer. Anal. 52 (2014), no 3., 1203--1221.
		
		\bibitem{hw16}
		Hohage T. and Werner F. {\it Inverse Problems with Poisson Data: statistical regularization theory, applications and algorithms.} Inverse Problems 32 (2016), no. 9, 093001.
		
		\bibitem{hw17}
		Hohage T. and Weidling F. {\it Characterizations of variational source conditions, converse results, and maxisets of spectral regularization methods.} SIAM J. Numer. Anal. 55(2): 598-620, 2017.
		
		\bibitem{ILS2013}
		Iglesias M. A.; Law K.; Stuart A. M.: {\it Ensemble Kalman methods for inverse problems.} Inverse Problems 29 (2013), no. 4, 045001, 20 pp.
		
		\bibitem{ILLS2017}
		Iglesias M. A.; Lin K.; Lu S. and Stuart A.~M.: {\it Filter based methods for statistical linear inverse problems}. Commun. Math. Sci. 15 (2017), no. 7, 1867--1895.
		
		\bibitem{KP2018}
		Kaltenbacher B. and Pedretscher B. {\it Parameter estimation in SDEs via the Fokker-Planck equation: likelihood function and adjoint based gradient computation}. J. Math. Anal. Appl. 465 (2018), no. 2, 872--884.
		
		\bibitem{LSZ2015}
		Law K.; Stuart A. and Zygalakis K. {\it Data assimilation. A mathematical introduction}. Texts in Applied Mathematics, 62. Springer, Cham, 2015. xviii+242 pp.
		
		\bibitem{lw20}
		Li H. and Werner F. {\it Empirical Risk Minimization as Parameter Choice Rule for General Linear Regularization Methods}. Annales de l'Institut Henri Poincar\'e 56 (2020), no. 1, 405--427.
		
		\bibitem{LLM}
		Lin K.; Lu S. and Math\'{e} P. {\it Oracle-type posterior contraction rates in Bayesian inverse problems}. Inverse Probl. Imaging 9 (2015), no. 3, 895--915.
		
		\bibitem{LP2013}
		Lu S. and Pereverzev S.~V. {\it Regularization theory for ill-posed problems. Selected topics}. Inverse and Ill-posed Problems Series, 58. De Gruyter, Berlin, 2013. xiv+289 pp.
		
		\bibitem{MP2003}
		Math\'{e} P. and Pereverzev S.~V. {\it Geometry of linear ill-posed problems in variable Hilbert scales}. Inverse Problems 19 (2003), no. 3, 789--803.
		
		\bibitem{MP2006}
		Math\'{e} P. and Pereverzev S.~V. {\it Regularization of some linear ill-posed problems with discretized random noisy data.} Math. Comp. 75 (2006), no. 256, 1913--1929.
		
		\bibitem{msw20}
		Munk A.; Staudt T. and Werner F. {\it Statistical foundations of nanoscale photonic imaging}. In: Nanoscale photonic imaging, to appear, Springer 2020.
		
		\bibitem{SS2017}
		Schillings C. and Stuart A.~M. {\it Analysis of the ensemble Kalman filter for inverse problems}. SIAM J. Numer Anal. 2017;55(3):1264--1290.
		
		\bibitem{stuart2010}
		Stuart A.~M. {\it Inverse problems: a {B}ayesian perspective.} Acta Numer.19 (2010), 451--559.
		
		\bibitem{srs17}
		Stanhope S.; Rubin J. E. and Swigon D. {\it Robustness of solutions of the inverse problem for linear dynamical systems with uncertain data}. SIAM/ASA J. Uncertain. Quantif. 5 (2017), no. 1, 572--597.
		
		\bibitem{T1994}
		Tautenhahn U. {\it On the asymptotical regularization of nonlinear ill-posed problems}. Inverse Problems 10 (1994), no. 6, 1405--1418.
		
		\bibitem{w12}
		Werner F. {\it Inverse problems with {P}oisson data: {T}ikhonov-type regularization and iteratively regularized {N}ewton methods}.
		\newblock PhD thesis, University of {G\"o}ttingen, 2012.
		\newblock http://num.math.uni-goettingen.de/\~{}f.werner/files/diss\_frank\_werner.pdf.
		
		\bibitem{w18}
		Werner F. {\it Adaptivity and Oracle Inequalities in Linear Statistical Inverse Problems: a (numerical) survey.} In: New Trends in Parameter Identification for Mathematical Models, 291-316, Birkh\"auser, 2018.
		
		\bibitem{wh20}
		Werner F. and Hofmann B. {\it Convergence Analysis of (Statistical) Inverse Problems under Conditional Stability Estimates.}
		\newblock Inverse Problems 36 (2020), no. 1, 015004.
		
		\bibitem{ZH2018}
		Zhang Y. and Hofmann B. {\it On the second-order asymptotical regularization of linear illposed
			inverse problems}. to appear at Applicable Analysis, https://doi.org/10.1080/00036811.2018.1517412.
		
	\end{thebibliography}
\end{document}